\documentclass[reqno,twoside]{amsart}
\usepackage{amsmath,amssymb,colordvi}
\usepackage{multicol}
\usepackage{color,bbm}
\usepackage{epic}
\usepackage{pstricks}
\usepackage{graphicx}
\usepackage{xcolor}

\usepackage{a4wide,amsmath,amssymb,latexsym,amsthm}
\usepackage{graphicx}

\newtheorem{theorem}{Theorem}[section]
\newtheorem{proposition}[theorem]{Proposition}
\newtheorem{remark}[theorem]{Remark}
\newtheorem{lemma}[theorem]{Lemma}
 
\newtheorem{definition}[theorem]{Definition}

\numberwithin{equation}{section}

\graphicspath{{./Imgs/}}

\newcommand{\R}{\mathbb R}

\newcommand{\be}{\begin{equation}}
\newcommand{\ee}{\end{equation}}
\newcommand{\ba}{\begin{eqnarray}}
\newcommand{\ea}{\end{eqnarray}}
\newcommand{\beq}{\begin{equation}}
\newcommand{\eeq}{\end{equation}}

\usepackage{color}
\definecolor{red}{rgb}{0,0,0}

\usepackage[textsize=small]{todonotes}

\numberwithin{equation}{section}

\def\Omc{\R^N\setminus \Omega}

\def\Omb{\overline{\Omega}}

\def\RR{{\mathbb{R}}}

\def\Om{\Omega}

\def\bOm{\overline{\Omega}}

\usepackage{color}
\usepackage{stmaryrd}

\newenvironment{dedication}


\begin{document}

\title[Turnpike property]{Exponential Turnpike property for fractional  parabolic equations with non-zero exterior data}

\author{Mahamadi Warma}
\address{M. Warma, Department of Mathematical Sciences and the Center for Mathematics and Artificial Intelligence (CMAI), George Mason University, VA 22030, USA.}
\email{mwarma@gmu.edu}

\author{Sebasti\'an Zamorano}
\address{S. Zamorano, Universidad de Santiago de Chile, Facultad de Ciencia, Departamento de
Matem\'atica y Ciencia de la Computaci\'on, Casilla 307-Correo 2,
Santiago, Chile.}
 \email{sebastian.zamorano@usach.cl}

\thanks{The work of the first author is partially supported by the Air Force Office of Scientific Research (AFOSR) under Award NO [FA9550-18-1-0242] and the Army Research Office (ARO) under Award NO:  W911NF-20-1-0115. The second author is supported by the Conicyt PAI Convocatoria Nacional Subvenci\'on a la Instalaci\'on en la  Academia Convocatoria 2019 PAI77190106.}

\begin{abstract}
We consider averages convergence as the time-horizon goes to infinity of optimal solutions of time-dependent optimal control problems to optimal solutions of the corresponding stationary optimal control problems.
Control problems play a key role in engineering, economics and sciences. To be more precise, in climate sciences, often times, relevant problems are formulated in long time scales, so that, the problem of possible asymptotic behaviors when the time-horizon goes to infinity becomes natural. Assuming that the controlled dynamics under consideration are stabilizable towards a stationary solution, the following natural question arises: Do time averages of optimal controls and trajectories converge to the stationary optimal controls and states as the time-horizon goes to infinity?
This question is very closely related to the so-called turnpike property that shows that, often times, the optimal trajectory joining two points that are far apart, consists in, departing from the point of origin, rapidly getting close to the steady-state (the turnpike) to stay there most of the time, to quit it only very close to the final destination and time.
In the present paper we deal with heat equations with non-zero exterior conditions (Dirichlet and nonlocal Robin) associated with the fractional Laplace operator $(-\Delta)^s$ ($0<s<1$). We prove the turnpike property for the nonlocal Robin optimal control problem and the exponential turnpike property for both Dirichlet and nonlocal Robin optimal control problems.
\end{abstract}

\maketitle

\begin{dedication}
\begin{center}
{\em Dedicated to Professor Enrique Zuazua on the occasion of his $60^{th}$ birthday.}\\
\end{center}

Besides the controllablity properties of evolution equations, the turnpike property of optimal control problems and its applications to life science and industry are also an important part of the exceptional contributions of Enrique Zuazua in Partial Differential Equations (PDE) and  applied mathematics. He has dedicated several years of his research activities to apply these concepts to many problems of interest and has motivated young and future generations of mathematicians to develop new tools that have concrete applications in life science, industry and engineering. 

I (Mahamadi) personally met Enrique for the first time in 2008 at the University of Puerto Rico, Rio Piedras Campus (UPRRP), where he was the keynote speaker of the scientific activities of the College of Natural Sciences at UPRRP. Since then, Enrique visited the UPRRP each year in December for two weeks. In 2012, he initiated me to control theory of PDE and we interchanged on fractional PDE. Since then, we have maintained a sincere and fruitful collaboration in mathematics. We wrote several outstanding papers and now, it is even more, we are good friends. 

I visited the Basque Center of Applied Mathematics (BCAM) in Bilbao (Enrique is the founding Director) for one month in 2015, where I met for the first time Sebasti\'an who was doing a doctoral internship with Enrique (studying the turnpike property for the two-dimensional Navier-Stokes equations). This is the story between me, Enrique and Sebasti\'an. We would like to take the opportunity to put on record our sincere gratitude, appreciation and to dedicate this article in his honor. Happy birthday Enrique! \textexclamdown Feliz cumplea\~nos Enrique!  \textexclamdown Zorionak Enrique!
\end{dedication}

\section{Introduction}

In the present paper, we address the question of the limiting behavior of optimal control problems as the time-horizon goes to infinity (turnpike property) for fractional heat equations with inhomogeneous Dirichlet and nonlocal Robin exterior data.

The motivation to consider this kind of problem is clear in many contexts but in particular in climate sciences where problems are naturally formulated in long time intervals. This is for instance the case of paleoclimatology (study of past climates) where the problem of the inversion of past climates is addressed.

The concept of turnpike property of an optimal control problem for large enough time-horizon, roughly speaking, describes that the optimal nonstationary solution is made of three arcs: The first and the last being transient short time arcs, and the middle piece being a long--time arc staying exponentially close to the optimal steady-state of the associated static optimal control problem. This property was introduced a long time ago in the field of econometry for discrete time optimal control problems in finite dimension (see e.g. \cite{dorfman1958linear, mckenzie1963turnpike}). Again in econometry, the turnpike phenomena also appears in model predictive control problems (\cite{faulwasser2015design, grune2013economic}). We refer to the monographs \cite{zaslavski2006turnpike, zaslavski2019turnpike} and the references therein for a complete overview of turnpike properties for a variety of systems.

One kind of turnpike phenomena is the so--called exponential turnpike property. In this case, not only the optimal state and control, but also the corresponding adjoint vectors remain exponentially close to the stationary optimal control, state and adjoint vectors for a large enough time-horizon. In the case of finite--dimensional systems we refer to \cite{porretta2013long}, where the exponential turnpike property has been proven under the Kalman rank condition. For the nonlinear finite dimensional setting  we refer to \cite{trelat2015turnpike}. In \cite{porretta2013long}, a rigorous analysis of the extremal equations has been done for linear infinite dimensional systems under  suitable assumptions of observability, and these results have been extended to the case of semilinear heat equations in \cite{porretta2016remarks}. Both works \cite{porretta2013long, porretta2016remarks} have shown the exponential turnpike property by using the fact that the extremal equations can be decoupled (see \cite[Chapter III]{lions1971optimal}), and by employing the algebraic Riccati equation associated to this decoupling. In \cite{grune2019sensitivity} the authors proved the exponential turnpike property without using the Riccati theory, but under assumptions on stabilizability and detectability. All the above mentioned works have considered optimal control problems without terminal constrains with the exception of \cite{breiten2020onthe}, where terminal constraints habe been studied. In \cite{trelat2018steady}, the authors have investigated the turnpike property in Hilbert spaces by using general semigroups method with bounded controls and observation operators, and for boundary control parabolic equations. Besides, other contributions taking into account the turnpike analysis are contained in \cite{hernandez2019greedy, trelat2018integral, zamorano2018turnpike}.  Finally, we refer to \cite{grune2020exponential} for a turnpike analysis of general evolution equations.

Enrique Zuazua has made an exceptional contribution in the topic of turnpike properties of evolution equations and its applications to life science and industry. He has thoroughly studied the turnpike property for finite-dimensional linear and nonlinear optimal control problems,   linear and semilinear heat equations, wave equations, optimal shape design,  steady-state and  periodic exponential turnpike property of optimal control problems in Hilbert spaces.  
We refer for instance to his works \cite{esteve2020turnpike,gugat2016optimal,hernandez2019greedy,gontran2020shape,LANCE2019496,porretta2013long,porretta2016remarks,sakamoto2019turnpike,trelat2015turnpike,trelat2018steady,ZUAZUA2017199}
for more details.

Let us notice that all the mentioned works deal with the traditional approaches where the control is localized in the interior of the domain or on a part of the boundary, as long as, the control operator is bounded (\cite{grune2019sensitivity, porretta2013long, trelat2018steady}), or it is unbounded but with an admissible control (\cite{grune2020exponential}). However, in many real life applications the control can be located outside the domain where the PDE is satisfied, which is not surprising, since in many cases we do not have a direct access to the boundary of the domain. This can occur in the following situations (but not limited to):  Acoustic testing, when the loudspeakers are placed far from the aerospace structures \cite{larkin1999direct};  Magnetotellurics (MT), which is a technique to infer earth's subsurface electrical conductivity from surface measurements \cite{unsworth2005new, CWeiss_BvBWaanders_HAntil_2018a};  Magnetic Drug Targeting (MDT), where drugs with ferromagnetic particles in suspension are injected into the body and the external magnetic field is then used to steer the drug to relevant areas, for example, solid tumors \cite{HAntil_RHNochetto_PVenegas_2018a, HAntil_RHNochetto_PVenegas_2018b,lubbe1996clinical}; and Electroencephalography (EEG) is used to record electrical activities in brain \cite{niedermeyer2005electroencephalography,williams1974electroencephalography}, 
in case one accounts for the neurons disjoint from the brain, one will obtain an external control problem. We also refer to \cite{antil2019external, antil2020external}  for other relevant applications.

Recently, in \cite{antil2019external, antil2020external} the authors introduced the notion of external optimal control problems with elliptic and parabolic space--fractional PDE as constraints. They considered a nonlocal diffusion operator such as the  fractional Laplacian $(-\Delta)^s$, with $0<s<1$, which allows to replace the classical boundary condition by an exterior datum. 


In this work we are interested to the  turnpike property for fractional parabolic equations with Dirichlet and nonlocal Robin type external controls. We will give a complete analysis of the relationship between the optimal solution of the following time--dependent exterior optimal control problem:
\begin{align}\label{eqmain1}
\min_{g\in U} J^T(g):=\frac{1}{2}\int_0^T\|u-u^d\|_{L^2(\Om)}^2dt + \frac{1}{2}\int_0^T\|g(\cdot,t)\|_{L^2(\Omc,\mu)}^2 dt,
\end{align}
subject to the constraints that $u$ solves the fractional heat equation
\begin{align}\label{eqmain1-1}
\begin{cases}
u_t+(-\Delta)^s u=0 & \text{ in }\Omega\times(0,T),\\
\mathcal{B}u=\beta g & \text{ in } (\Omc)\times(0,T),\\
u(\cdot, 0) =0 & \text{ in }\Omega,
\end{cases}
\end{align}
and the corresponding stationary problem, that is, 
\begin{align}\label{eqmain2}
\min_{g\in \mathcal U} J(g) :=\frac{1}{2}\|u-u^d\|_{L^2(\Omega)}^2 + \frac{1}{2}\|g\|_{L^2(\Omc,\mu)}^2,
\end{align}
subject to the constraints that  $u$ solves the fractional elliptic equation
\begin{align}\label{eqmain2-1}
\begin{cases}
(-\Delta)^s u=0 & \text{ in }\Omega,\\
\mathcal{B} u =  \beta g & \text{ in }\Omc.
\end{cases}
\end{align}
Here, $\Omega\subset\RR^N$ is a bounded open set with a Lipschitz continuous boundary $\partial\Omega$, $T>0$ is a real number,  and $u^d\in L^2(\Omega)$ is a fixed target. 
 In addition, $0<s<1$ and $(-\Delta)^s$ denotes the fractional Laplace operator given formally for a smooth function $u$ by the following singular integral:
\begin{align*}
(-\Delta)^s u(x):=C_{N,s} \text{P.V.}\int_{\R^N} \frac{u(x)-u(y)}{|x-y|^{N+2s}}dy, \quad x\in \R^N,
\end{align*}
where $C_{N,s}$ is a normalization constant. We  refer to Section \ref{sec-2} for the precise definition. 

\begin{itemize}
\item In the case of the Dirichlet problem, the operator $\mathcal{B}$ is given by $\mathcal{B}u=u$, $\beta=1$ and the function $g\in L^2((0,T);L^2(\Omc))$. For the time-dependent problem \eqref{eqmain1} the Banach space $U:=L^2((0,T);L^2(\Omc))$ and $\mathcal U:=L^2(\Omc)$ in \eqref{eqmain2}. 

\item  For the nonlocal Robin problem, $\mathcal{B}u=\mathcal{N}_su +\beta u$ , where $\mathcal{N}_s$ is the nonlocal normal derivative of $u$ (see Section \ref{sec-2})  and  $\beta\in L^1(\Omc)$ is a given non--negative function. In that case,  $U:=L^2((0,T);L^2(\Omc,\mu))$ and $\mathcal U:=L^2(\Omc,\mu)$, where the measure $\mu$ on $\Omc$ is defined by $d\mu:=\beta dx$ with $dx$ being the usual $N$-dimensional Lebesgue measure.
\end{itemize}

It has been recently shown in \cite{antil2020external} that the problem \eqref{eqmain1}-\eqref{eqmain1-1}  is well-posed in the sense that,  there exist one optimal pair $(g^T,u^T)\in L^2((0,T);L^2(\Omc,\mu))\times\Big( L^2((0,T); H_{\Omega,\beta}^s(\Omega))\cap H^1((0,T); (H_{\Omega,\beta}^s)^{*})\Big)$, for the Robin conditon, and $(g^T,u^T)\in L^2((0,T);L^2(\Omc))\times L^2((0,T);L^2(\Omega))$ for the Dirichlet condition, which are the optimal solutions of the minimization problems. Similarly, by \cite{antil2019external} the minimization problems \eqref{eqmain2}-\eqref{eqmain2-1} have one solution $(\overline{g},\overline u)\in L^2(\Omc,\mu)\times H_{\Omega,\beta}^s(\Omega)$ for the Robin problem and one solution $(\overline{g},\overline u)\in L^2(\Omc)\times L^2(\Omega)$ for the Dirichlet problem. We refer to Section \ref{sec-2} for the definition of the involved spaces and to Sections \ref{sec-4} and \ref{sec-5} for more details.

The main concern of the present article is to investigate if there exists any connection between the optimal pairs $(g^T,u^T)$ and $(\overline{g},\overline{u})$, when the time-horizon $T$ is sufficiently large.

To the best of our knowledge, it is the first time that the turnpike property is studied for the fractional Laplace operator with non-zero Dirichlet and/or nonlocal Robin exterior data.

The key difficulties and novelties of the present paper can be summarized as follows:
\begin{enumerate}
\item[1.]  From the definition of the fractional Laplacian $(-\Delta)^s$, we easily see that it is a nonlocal operator. That is, in order to evaluation $(-\Delta)^su(x)$ at a point $x$, it is necessary to have  information on $u$ over the whole space $\R^N$. Besides,  contrary to the local case of the Laplace operator, $(-\Delta)^s u$ may be nonsmooth even if the function $u$ is smooth (see e.g. \cite{RS-DP}).

\item[2.] Since we are considering exterior data $g\in L^2$, the associated Dirichlet problems (stationary and time-dependent) only admit solutions by transposition (very weak solutions) which are not smooth enough. In addition, for both parabolic problems (Dirichlet and Robin exterior data), one cannot use directly semigroups method to show the existence of solutions (since the involved operator is in general not a generator of a semigroup).


\item[3.] To obtain the turnpike property for the Dirichlet problem, it is necessary to use the notion of admissible control and observation operators (see e.g. \cite{tucsnak2009observation,tucsnak2014well} for the local case). For the state problem with Dirichlet exterior conditions, since we are considered exterior data only in $L^2$, with this concept, we will exploit semigroups theory to prove the existence and uniqueness of solutions to the state equation. In addition, it will help to obtain an extra regularity in time for the optimal state but losing the space--regularity, that is, $u^T\in C([0,T]; H^{-s}(\Omega))\cap L^2(\Om\times(0,T))$. This continuity property of solutions is necessary to obtain the exponential turnpike property.  As far as we know, this is the first article dealing with these concepts in the case of fractional evolution equations.

\item[4.] For the Robin problem, we have shown the convergence of solutions of the finite horizon control problems in $(0,T)$ to their corresponding steady state versions as the time-horizon $T$ tends to infinity (turnpike property). We have also obtained the exponential turnpike property.

\item[5.] In the case of the Dirichlet exterior control problem, using the concept of admissible control and observation operators  we have established the exponential turnpike property for the corresponding systems (state, control and adjoint vectors). 
\end{enumerate}

The rest of the paper is organized as follows. In Section \ref{sec-2} we introduce the function spaces needed to study our problems and give some intermediate known results that are needed in the proof of our main results. Section \ref{sec-4} contains some recent known results on the Robin optimal control problems.
 Section \ref{subsec4} is devoted to the proof of the main results for the Robin problem, that is, the turnpike and exponential turnpike properties. These results are contained in Theorems \ref{mainresult1} and \ref{mainresult2}, respectively. 
 Section \ref{sec-5} contains some known results for the Dirichlet control problem. In Section \ref{sec42} we rewrite this problem as an abstract Cauchy problem by using the notion of admissible control operators. 
 Finally, in Section \ref{subsec-6} we prove the exponential turnpike property of the Dirichlet optimal control problem, namely, Theorem \ref{mainresult3}.

\section{Preliminary results}\label{sec-2}
In this section we give some notations and recall some known results as they are needed  throughout the paper.
We start with fractional order Sobolev spaces. 

For $0<s<1$ a real number and $\Omega\subset\RR$ an arbitrary open set,  we let
\begin{align*}
H^{s}(\Omega):=\left\{u\in L^2(\Omega):\;\int_\Omega\int_\Omega\frac{|u(x)-u(y)|^2}{|x-y|^{N+2s}}\;dxdy<\infty\right\},
\end{align*}
and we endow it with the norm defined by
\begin{align*}
\|u\|_{H^{s}(\Omega)}:=\left(\int_\Omega|u(x)|^2\;dx+\int_\Omega\int_\Omega\frac{|u(x)-u(y)|^2}{|x-y|^{N+2s}}\;dxdy\right)^{\frac 12}.
\end{align*}
We set
\begin{align*}
H_0^{s}(\Omega):=\Big\{u\in H^{s}(\R^N):\;u=0\;\mbox{ a.e. in }\;\R^N\setminus \Omega\Big\}.
\end{align*}
We notice that if $\Omega$ is bounded and $0<s\ne 1/2<1$, then $H_0^s(\Omega)=\overline{\mathcal D(\Omega)}^{H^1(\Omega)}$, where $\mathcal D(\Omega)$ denotes the space of all continuously infinitely differentiable functions with compact support in $\Omega$.

We shall denote by $ H^{-s}(\Omega):=( H_0^{s}(\Omega))^\star$ the dual space of $H_0^{s}(\Omega)$ with respect to the pivot space $L^2(\Om)$,  so that we have the following continuous embeddings: $$H_0^{s}(\Omega) \hookrightarrow L^2(\Omega)\hookrightarrow H^{-s}(\Omega).$$

For more information on fractional order Sobolev spaces, we refer to \cite{NPV,Gris,War} and their references.

Next, we give a rigorous definition of the fractional Laplace operator. To do this, we need the following function space:
\begin{align*}
\mathcal L_s^{1}(\R^N):=\left\{u:\R^N\to\R\;\mbox{ measurable and }\; \int_{\R^N}\frac{|u(x)|}{(1+|x|)^{N+2s}}\;dx<\infty\right\}.
\end{align*}
For $u\in \mathcal L_s^{1}(\R^N)$ and $\varepsilon>0$ we set
\begin{align*}
(-\Delta)_\varepsilon^s u(x):= C_{N,s}\int_{\{y\in\R^N:\;|x-y|>\varepsilon\}}\frac{u(x)-u(y)}{|x-y|^{N+2s}}\;dy,\;\;x\in\R^N.
\end{align*}
Here, $C_{N,s}$ is a normalization constant given by
\begin{align}\label{CNs}
C_{N,s}:=\frac{s2^{2s}\Gamma\left(\frac{2s+N}{2}\right)}{\pi^{\frac{N}{2}}\Gamma(1-s)},
\end{align}
where $\Gamma$ is the usual Gamma function.

The {\em fractional Laplacian}  $(-\Delta)^su$ is defined for $u\in \mathcal L_s^{1}(\R^N)$ by the following singular integral:
\begin{align}\label{fl_def}
(-\Delta)^su(x):=C_{s}\,\mbox{P.V.}\int_{\R^N}\frac{u(x)-u(y)}{|x-y|^{N+2s}}\;dy=\lim_{\varepsilon\downarrow 0}(-\Delta)_\varepsilon^s u(x),\;\;x\in\R^N,
\end{align}
provided that the limit exists for a.e. $x\in\R^N$. 
For more details on the fractional Laplace operator we refer to \cite{Caf3,NPV,GW-CPDE,War} and their references.

Assume that $\Omega\subset\R^N$ is a bounded open set with a Lipschitz continuous boundary $\partial\Omega$.

We consider the realization of $(-\Delta)^s$ in $L^2(\Omega)$ with a Dirichlet exterior condition $u=0$ in $\Omc$. More precisely, we consider 
the selfadjoint operator $(-\Delta)_D^s$ on $L^2(\Omega)$ given by
\begin{equation}\label{DLO}
D((-\Delta)_D^s):=\Big\{u\in H_0^{s}(\Om):\; (-\Delta)^su\in L^2(\Omega)\Big\},\;\;\;
(-\Delta)_D^su:=((-\Delta)^su)|_{\Omega}.
\end{equation}
It is well-know (see e.g. \cite{C-Wa}) that the operator $-(-\Delta)_D^s$ generates a strongly continuous submarkovian semigroup (positivity-preserving and $L^\infty$-contractive) $(e^{-t(-\Delta)_D^s})_{t\ge 0}$ on $L^2(\Omega)$.

For $\beta\in L^1(\Omc)$ a given non-negative function, we denote by $H_{\Omega,\beta}^s$ the following space:
\begin{align*}
H_{\Omega,\beta}^s:=\Big\{u:\R^N\to \R \text{ measurable and }  \|u\|_{H_{\Omega,\beta}^s}<\infty\Big\},
\end{align*}
where 
\begin{align}\label{normspace}
\|u\|_{H_{\Omega,\beta}^s}:=\left(\|u\|_{L^2(\Omega)}^2+\|\beta^{1/2}u\|_{L^2(\Omc)}^2+\int\int_{\R^{2N}\setminus(\R^N\setminus\Omega)^2}\frac{|u(x)-u(y)|^2}{|x-y|^{N+2s}}dxdy\right)^{\frac{1}{2}},
\end{align}
 and 
\begin{align*}
\R^{2N}\setminus(\R^{2N}\setminus\Omega)^2=(\Omega\times\Omega)\cup((\Omc)\times\Omega)\cup(\Omega\times(\Omc)).
\end{align*}

Let $\mu$ be the measure on $\Omc$ given by $d\mu=\beta dx$. Then, the norm \eqref{normspace} can be rewritten as
\begin{align}\label{normspace2}
\|u\|_{H_{\Omega,\beta}^s}=\left(\|u\|_{L^2(\Omega)}^2+\|u\|_{L^2(\Omc,\mu )}^2+\int\int_{\R^{2N}\setminus(\R^N\setminus\Omega)^2}\frac{|u(x)-u(y)|^2}{|x-y|^{N+2s}}dxdy\right)^{\frac{1}{2}}.
\end{align}

When $\beta=0$, we shall let $H_{\Omega}^s:=H_{\Omega,0}^s$.  

It has been shown in \cite[Proposition 3.1]{DRV} that for every $\beta\in L^1(\Omc)$, $H_{\Omega,\beta}^s$ is a Hilbert space. We shall denote by $(H_{\Omega,\beta}^s)^{*}$ the dual space of $H_{\Omega,\beta}^s$.

Next, for $u\in H_{\Omega}^s$ we introduce the {\em nonlocal normal derivative $\mathcal N_su$} of $u$ defined by 
\begin{align}\label{NLND}
\mathcal N_su(x):=C_{N,s}\int_{\Omega}\frac{u(x)-u(y)}{|x-y|^{N+2s}}\;dy,\;\;\;x\in\R^N \setminus\bOm,
\end{align}
where $C_{N,s}$ is the constant given in \eqref{CNs}.
We notice that since equality is to be understood a.e., we have that \eqref{NLND} is the same as  a.e.  in $\Omc$.
By \cite[Lemma 3.2]{GSU},  for every $u\in H_{\Omega}^s$, we have that $\mathcal N_su\in H_{\rm loc}^s(\Omc)$. In addition, if $(-\Delta)^su\in L^2(\Omega)$, then $\mathcal N_su\in L^2(\Omc)$.
The operator $\mathcal N_s$ has been called "interaction operator" in  \cite{antil2019external,QDu_MGunzburger_RBLehoucq_KZhou_2013a}. Several properties of $\mathcal N_s$ have been studied in \cite{C-Wa,DRV}. 

%


We have the following integration by parts formula.

\begin{lemma}
Let $u\in H_{\Omega}^s$ be such that $(-\Delta)^s u\in L^2(\Omega)$ and $\mathcal N_su\in L^2(\Omc)$. Then for every $v\in H^s(\RR^N)$, the identity
\begin{align}\label{Int-Part}
\frac{C_{N,s}}{2}\int\int_{\R^{2N}\setminus(\R^N\setminus\Omega)^2}&\frac{(u(x)-u(y))(v(x)-v(y))}{|x-y|^{N+2s}}\;dxdy\notag\\
=&\int_{\Omega}v(x)(-\Delta)^su(x)\;dx+\int_{\Omc}v(x)\mathcal N_su(x)\;dx,
\end{align}
holds.
\end{lemma}

We refer to \cite[Lemma 3.3]{DRV} (see also \cite[Proposition 3.7]{War-ACE}) for the proof and more details.

We mention that if $u\in H_0^s(\Omega)$ or $v\in H_0^s(\Omega)$, then
\begin{align*}
\int\int_{\R^{2N}\setminus(\R^N\setminus\Omega)^2}\frac{(u(x)-u(y))(v(x)-v(y))}{|x-y|^{N+2s}}\;dxdy=\int_{\R^N}\int_{\R^{N}}\frac{(u(x)-u(y))(v(x)-v(y))}{|x-y|^{N+2s}}\;dxdy.
\end{align*}

Throughout the remainder of the article, for $u,v\in H_{\Omega,\beta}^s$, we shall denote
\begin{align}\label{bilinear}
\mathcal E(u,v):=\frac{C_{N,s}}{2}\int\int_{\R^{2N}\setminus(\R^{2N}\setminus\Omega)^2}\frac{(u(x)-u(y))(v(x)-v(y))}{|x-y|^{N+2s}}dxdy+\int_{\Omc} \beta uv dx,
\end{align}
where $\beta\in L^1(\Omc)$ is a given non-negative function.

We observe that the form $\mathcal E$ is bilinear and continuous.

Next, we consider the following fractional elliptic problem:
\begin{align}\label{NonlocalPDE}
\begin{cases}
(-\Delta)^s u =f & \text{ in }\Omega,\\
\mathcal{N}_s u +  \beta u= \beta g & \text{ in }\Omc,
\end{cases}
\end{align}
where $\mathcal{N}_s$ is the nonlocal normal derivative introduced in \eqref{NLND}.

Let $g\in L^2(\Omc,\mu)$ and $f\in (H_{\Omega,\beta}^s)^{*}$. We say that a function $u\in H_{\Omega,\beta}^s$ is a weak solution of \eqref{NonlocalPDE} if the identity 
\begin{align}\label{weaksol2}
\mathcal E(u,v)=\langle f,v\rangle_{(H_{\Omega,\beta}^s)^{*},H_{\Omega,\beta}^s} +\int_{\Omc}  gv d\mu,
\end{align}
holds, for every $v\in H_{\Omega,\beta}^s$, where we recall that $\mathcal E$ is given in \eqref{bilinear}.

Using the classical Lax-Milgram lemma, it is easy to show that, for every $g\in L^2(\Omc,\mu)$ and $f\in (H_{\Omega,\beta}^s)^{*}$, there exists a unique weak solution $u\in H_{\Omega,\beta}^s$ of \eqref{NonlocalPDE}.


We conclude this section by introducing the realization in $L^2(\Omega)$ of $(-\Delta)^s$ with the nonlocal Robin exterior condition.

For a function $u\in L^2(\Om)$ we define its extension \(u_R\) as follows:
\begin{align*}
    u_R(x):=\begin{cases} u(x) &\text{ if } x \in \Omega, \\
 \displaystyle   \frac{C_{n,s}}{C_{N,s}\rho(x)+\beta(x)}\int_\Omega \frac{u(y)}{\vert x-y\vert^{N+2s}} dy &\text{ if } x\in \R^N\setminus \Omb,
    \end{cases}
\end{align*}
where the function $\rho$ is given by
\begin{align*}
\rho(x):=\int_{\Om}\frac{1}{|x-y|^{N+2s}}\;dy,\;\;x\in\RR^N\setminus \Omb.
\end{align*}
Then, $u_R$ is well defined for every $u\in L^2(\Omega)$.

Let $u\in H_\Om^{s}$. It has been shown in \cite{C-Wa} that  $u_R$ satisfies the following nonlcoal Robin exterior condition:
\begin{align}\label{RC}
\mathcal N_su_R+\beta u_R=0\;\;\mbox{ in }\;\R^N\setminus \Omb.
\end{align}

Let 
\begin{align*} 
    D(\mathcal E_R):=\Big\{ u \in L^2(\Om):\;u_R \in H^{s}_{\Omega,\beta}\Big \},
\end{align*}
and  \(\mathcal E_R :D(\mathcal E_R)\times D(\mathcal E_R) \rightarrow \mathbb{R}\) be given by
\begin{align*} 
    \mathcal E_R(u,v):=\frac{C_{N,s}}{2}\int\int_{\R^{2N}\setminus(\R^N\setminus\Omega)^2}\frac{(u_R(x)-u_R(y))(v_R(x)-v_R(y))}{|x-y|^{N+2s}}\;dxdy+\int_{\Omc} \beta u_R v_R\;dx.
\end{align*}
Then,  \(\mathcal E_R\) is a closed, symmetric and densely defined bilinear form on $L^2(\Om)$.  The  selfadjoint operator $(-\Delta)_R^s$ on $L^2(\Om)$ associated with \(\mathcal E_R\) is given by 
\begin{equation} \label{op-R}
\begin{cases}
\displaystyle D((-\Delta)_R^s):=
   \Big\{ u \in L^2(\Omega):
    u_R \in H^{s}_{\beta,\Omega}\;\;\exists\; f\in L^2(\Omega)\mbox{ such that }\; u_R \mbox{ is a weak solution of } \eqref{NonlocalPDE} \\
    \hfill\text{ with right hand side } f \;\mbox{ and }\; g=0\Big\},\\
    (-\Delta)_R^s u:= f.
    \end{cases}
\end{equation}
Here also, we have that the operator $-(-\Delta)_R^s$ generates a strongly continuous submarkovian semigroup $(e^{-t(-\Delta)_R^s})_{t\ge 0}$ on $L^2(\Omega)$.
We refer to \cite{C-Wa} and their references for more information, details and qualitative properties of the operator  $(-\Delta)_R^s$.

\section{ Robin exterior control problems: The turnpike property}\label{sec4}

In this section we state and prove our main results concerning the turnpike property of the optimal control problems for  nonlocal Robin exterior data. In order to do this we need some preparations.

Throughout the following, without any mention, $\Omega\subset\R^N$ is a bounded open set with a Lipschitz continuous boundary $\partial\Omega$, $\beta\in L^1(\Omc)$ is a given non-negative function and the measure $\mu$ on $\Omc$ is given by $d\mu=\beta dx$. Given $T>0$, we denote $Q:=\Omega\times(0,T)$ and $\Sigma:=(\Omc)\times(0,T)$. Given a Banach space $\mathbb X$ and its dual $\mathbb X^\star$, we denote their duality pairing by $\langle\cdot,\cdot\rangle_{\mathbb X^\star,\mathbb X}$. If $\mathcal H$ is a Hilbert space, we  denote by $(\cdot,\cdot)_{\mathcal H}$ the scalar product in $\mathcal H$. If $X$ and $Y$ are two Banach spaces and $T:X\to Y$ is a bounded operator, we let $\|\cdot\|_{\mathcal L(X,Y)}$ ($\|\cdot\|_{\mathcal L(X)}$ if $X=Y$)  be the operator norm of $T$. Recall that the bilinear form $\mathcal E$ is given by
\begin{align*} 
    \mathcal E(u,v):=&\frac{C_{N,s}}{2}\int\int_{\R^{2N}\setminus(\R^N\setminus\Omega)^2}\frac{(u(x)-u(y))(v(x)-v(y))}{|x-y|^{N+2s}}\;dxdy+\int_{\Omc} \beta u v\;dx,
\end{align*}
for $u,v\in H^s_{\Omega,\beta}$. Also, $(-\Delta)_R^s$ denotes the operator defined in \eqref{op-R}.

%


\subsection{Nonlocal Robin optimal control problems}\label{sec-4}
 We consider the following nonlocal heat equation with nonlocal Robin exterior conditions:
\begin{align}\label{ocpevol1}
\begin{cases}
u_t+(-\Delta)^s u =0& \text{ in }Q,\\
\mathcal{N}_s u +  \beta u= \beta g & \text{ in }\Sigma,\\
u(\cdot, 0)= 0 &\text{ in } \Omega,
\end{cases}
\end{align}
where $\mathcal{N}_s u$ is the nonlocal normal derivative introduced in \eqref{NLND}.




Our notion of solutions to \eqref{ocpevol1} is the following.

\begin{definition}\label{defweaksol}
Let $g\in L^2((0,T);L^2(\Omc,\mu))$. We say that a function $u\in L^2((0,T); H_{\Omega,\beta}^s)\cap H^1((0,T);(H_{\Omega,
\beta}^s)^{*})$ is a weak solution of \eqref{ocpevol1} if the identity,
\begin{align}\label{weaksol}
\langle u_t, v\rangle_{(H_{\Omega,\beta}^s)^{*},H_{\Omega,\beta}^s}  +\mathcal E(u,v)=\int_{\Omc} gv d\mu,
\end{align}
holds, for every $v\in H_{\Omega,\beta}^s$ and almost every $t\in(0,T)$.
\end{definition}

We have the following existence result taken from \cite[Theorem 3.10]{antil2020external}.

\begin{theorem}\label{existencesolution}
For every $g\in L^2((0,T);L^2(\Omc,\mu))$, there exists a unique weak solution $u$ of \eqref{ocpevol1} in the sense of Definition \ref{defweaksol}. 
\end{theorem}

%
%
%
%

Next, let us consider the time-dependent optimal control problem  and the corresponding stationary one. Namely, we consider the minimization problems:
\begin{align}\label{ocpevol}
\min_{g\in L^2((0,T); L^2(\Omc,\mu))} J^T(g):=\frac{1}{2}\int_0^T\|u(\cdot,t)-u^d\|_{L^2(\Om)}^2dt + \frac{1}{2}\int_0^T\|g(\cdot,t)\|_{L^2(\Omc,\mu)}^2 dt,
\end{align}
subject to $u\in L^2((0,T);H_{\Omega,\beta}^s)\cap H^1((0,T); (H_{\Omega,\beta}^s)^{*})$ solves the fractional heat equation \eqref{ocpevol1},
and
\begin{align}\label{ocpest}
\min_{g\in L^2(\Omc,\mu )}J(g):=\frac{1}{2}\|u-u^d\|_{L^2(\Omega)}^2 + \frac{1}{2}\|g\|_{L^2(\Omc,\mu)}^2 ,
\end{align}
subject to $u\in H_{\Omega,\beta}^s$ solves the elliptic problem
\begin{align}\label{ocpest1}
\begin{cases}
(-\Delta)^s u=0 & \text{ in }\Omega,\\
\mathcal{N}_s u +\beta  u =\beta  g& \text{ in }\Omc,
\end{cases}
\end{align}
where $u^d\in L^2(\Omega)$ is a fixed target.

We have the following well--posedness results concerning problems \eqref{ocpevol1}-\eqref{ocpevol} and \eqref{ocpest}-\eqref{ocpest1}.

 
\begin{theorem}{\cite[Theorem 4.1]{antil2020external}}\label{Robin-evol}
There exist a unique optimal control $g^T\in L^2((0,T);L^2(\Omc,\mu))$ and a state $u^T\in L^2((0,T);H_{\Omega,\beta}^s)\cap H^1((0,T);(H_{\Omega,\beta}^s)^{*})$ such that the functional $J^T$ attains its minimum at $g^T$, and $u^T$ is the corresponding unique solution of \eqref{ocpevol1} with exterior datum $g^T$.
\end{theorem}


\begin{theorem}{\cite[Theorem 5.1]{antil2019external}}\label{Robin-sta}
There exist a unique optimal control  $\overline{g}\in L^2(\Omc,\mu)$ and  $\overline{u}\in H_{\Omega,\beta}^s$ a solution of \eqref{ocpest1} associated to $\overline{g}$, such that the functional $J$ attains its minimum at $\overline{g}$.
\end{theorem}

Additionally, we have the following first order optimality conditions for both optimal control problems. 


\begin{theorem}{\cite[Theorem 4.3]{antil2020external}}\label{antil}
If $g^T$ is a minimizer of \eqref{ocpevol}, then the first order necessary optimality conditions are given by
\begin{align}\label{firstorderevol}
\Big(\psi^T+ g^T, g-g^T\Big)_{L^2((0,T);L^2(\Omc,\mu))}\geq 0,\quad \forall g\in L^2((0,T);L^2(\Omc,\mu)),
\end{align}
where $\psi^T\in L^2((0,T);D(-\Delta)_R^s)\cap H^1((0,T);L^2(\Omega))$ solves the following adjoint problem:
\begin{align}\label{ocpevol2}
\begin{cases}
-\psi_t^T+(-\Delta)^s \psi^T=u^T-u^d & \text{ in }Q,\\
\mathcal{N}_s \psi^T +\beta  \psi^T = 0 & \text{ in }\Sigma,\\
\psi(\cdot,T) =0 & \text{ in }\Omega.
\end{cases}
\end{align}
Moreover, \eqref{firstorderevol} is equivalent to
\begin{align*}
g^T=-\psi^T\Big|_{\Sigma}.
\end{align*}
\end{theorem}

Let us notice that the regularity $\psi^T\in L^2((0,T);D(-\Delta)_R^s)\cap H^1((0,T);L^2(\Omega))$ of solutions to \eqref{ocpevol2} has been proved in \cite{BWZ-P}.

\begin{theorem}{\cite[Theorem 5.3]{antil2019external}}\label{antil2}
If $\overline{g}$ is a minimizer of \eqref{ocpest}, then the first order necessary optimality conditions are given by
\begin{align}\label{firstorderest}
\Big(\overline{\psi}+ \overline{g}, g-\overline{g}\Big)_{L^2(\Omc,\mu)}\geq 0,\quad \forall g\in L^2(\Omc,\mu),
\end{align}
where $\overline{\psi}\in H_{\Omega,\beta}^s$ solves the following adjoint problem:
\begin{align}\label{ocpest2}
\begin{cases}
(-\Delta)^s \overline{\psi}=\overline{u}-u^d & \text{ in }\Omega,\\
\mathcal{N}_s \overline{\psi} + \beta \overline{\psi} = 0 & \text{ in }\Omc.
\end{cases}
\end{align}
Moreover, \eqref{firstorderest} is equivalent to
\begin{align*}
\overline{g}=-\overline{\psi}\Big|_{\Omc}.
\end{align*}
\end{theorem}

To conclude this section, we mention that the solution $\overline{\psi}$ of \eqref{ocpest2} also belongs to $D((-\Delta)_R^s)$.

\subsection{The turnpike property}\label{subsec4}

In this section we state and prove our main results concerning the  nonlocal Robin exterior data. 

For this purpose, let $(g^T,u^T)\in L^2((0,T);L^2(\Omc,\mu))\times \Big(L^2((0,T);H_{\Omega,\beta}^s)\cap H^1((0,T);(H_{\Omega,\beta}^s)^{*})\Big)$  and $(\overline{g},\overline{u})\in L^2(\Omc,\mu)\times H_{\Omega,\beta}^s$ be the optimal pairs for the evolutionary and stationary optimal control problems \eqref{ocpevol1}-\eqref{ocpevol} and 
\eqref{ocpest}-\eqref{ocpest1}, respectively (see Theorems \ref{Robin-evol} and \ref{Robin-sta}). It follows from Theorems \ref{antil} and \ref{antil2} that there exists a pair
\begin{align*}
(\psi^T,\overline{\psi})\in \Big(L^2((0,T);D((-\Delta)_R^s))\cap H^1((0,T);L^2(\Omega))\Big) \times H_{\Om,\beta}^s
\end{align*}
 such that $g^T=-\psi^T\Big|_{\Sigma}$, $\overline{g}=-\overline{\psi}\Big|_{\Omc}$, and we have the following optimality systems:
\begin{align}\label{main1}
\begin{cases}
u_t^T+(-\Delta)^s u^T=0 & \text{ in }Q,\\
\mathcal{N}_s u ^T+ \beta u^T =\beta g^T & \text{ in }\Sigma,\\
u^T(\cdot, 0) =0 & \text{ in }\Omega,\\
-\psi_t^T+(-\Delta)^s \psi^T=u^T-u^d & \text{ in }Q,\\
\mathcal{N}_s \psi^T +\beta \psi^T = 0 & \text{ in }\Sigma,\\
\psi(\cdot,T) =0 & \text{ in }\Omega,
\end{cases}
\end{align}
and
\begin{align}\label{main2}
\begin{cases}
(-\Delta)^s \overline{u}=0 & \text{ in }\Omega,\\
\mathcal{N}_s \overline{u} + \beta\overline{u} = \beta\overline{g} & \text{ in }\Omc,\\
(-\Delta)^s \overline{\psi}=\overline{u}-u^d & \text{ in }\Omega,\\
\mathcal{N}_s \overline{\psi} + \beta \overline{\psi} = 0 & \text{ in }\Omc.
\end{cases}
\end{align}

The following theorem is the first main result of the paper.

\begin{theorem}\label{mainresult1}
Let $(u^T, g^T, \psi^T)$ be the solution of \eqref{main1}  and $(\overline{u},\overline{g}, \overline{\psi})$ be the solution of the corresponding stationary problem \eqref{main2}. Then,
\begin{align*}
\frac{1}{T}\int_0^T g^T\;dt \quad \longrightarrow \quad \overline{g}, &\quad \text{ strongly in } L^2(\Omc,\mu)\text{ as }T\to +\infty,
\end{align*}
and
\begin{align*}
\frac{1}{T}\int_0^T u^T\;dt \quad \longrightarrow \quad \overline{u}, &\quad \text{ strongly in } L^2(\Omega)\text{ as }T\to +\infty.
\end{align*}
\end{theorem}

\begin{proof}
We divide the proof in several steps.\\

{\bf Step 1}: We claim that there is a constant $C>0$ (independent of $T$) such that
\begin{align}\label{obs1}
\|u^T(\cdot,T)\|_{H_{\Omega,\beta}^s}^2\leq C \left[\int_0^T\|g^T(\cdot,t)\|_{L^2(\Omc,\mu)}^2\;dt + \int_0^T \|u^T(\cdot,t)-u^d\|_{L^2(\Omega)}^2\;dt \right].
\end{align}
Indeed, from Definition \ref{defweaksol} of weak solutions, we have that
\begin{align}\label{W1}
\langle u_t^T, v\rangle_{(H_{\Omega,\beta}^s)^{*},H_{\Omega,\beta}^s}  +\mathcal E(u^T,v)=\int_{\Omc} g^Tv d\mu, \quad \forall v\in H_{\Omega,\beta}^s.
\end{align}
Taking $v:=u^T$ as a test function in \eqref{W1}, we obtain
\begin{align}\label{W2}
\frac{1}{2}\frac{d}{dt}\|u^T(\cdot,t)\|_{H_{\Omega,\beta}^s}^2+\mathcal E(u^T,u^T)=\int_{\Omc}g^Tu^Td\mu.
\end{align}
Applying Cauchy--Schwarz's inequality, and then Young's inequality, to the right hand side in \eqref{W2}, we get that for every $\varepsilon>0$,
\begin{align}\label{m1}
\frac{1}{2}\frac{d}{dt}\|u^T(\cdot,t)\|_{H_{\Omega,\beta}^s}^2+\mathcal E(u^T,u^T)\leq \frac{1}{2\varepsilon}\int_{\Omc}|g^T|^2d\mu + \frac{\varepsilon}{2}\int_{\Omc}|u^T|^2d\mu.
\end{align}
It follows from the definition of  $\mathcal E$ that 
\begin{align*}
\|u\|_{L^2(\Omc,\mu)}^2\le \mathcal E(u,u),\;\;\forall\; u\in H_{\Omega,\beta}^s.
\end{align*}
 Thus, we get from \eqref{m1} that
\begin{align}\label{W3}
\frac{1}{2}\frac{d}{dt}\|u^T(\cdot,t)\|_{H_{\Omega,\beta}^s}^2+\mathcal E(u^T,u^T)\leq \frac{1}{2\varepsilon}\|g^T(\cdot,t)\|_{L^2(\Omc,\mu)}^2 + \frac{\varepsilon}{2}\mathcal E(u^T,u^T).
\end{align}

Integrating \eqref{W3} over $(0,T)$, we can deduce that (recall that $u^T(\cdot,0)=0$)
\begin{align}\label{obs1-1}
\|u^T(\cdot,T)\|_{H_{\Omega,\beta}^s}^2+\left(1-\frac{\varepsilon}{2}\right)\int_0^T\mathcal E(u^T,u^T)\; dt\leq \frac{1}{2\varepsilon} \int_0^T\int_{\Omc}|g^T|^2\;d\mu dt.
\end{align}
Choosing $\varepsilon$ so that $1-\frac{\varepsilon}{2}>0$, we get from  \eqref{obs1} that there is a constant $C>0$ (independent of $T$) such that
\begin{align*} 
\|u^T(\cdot,T)\|_{H_{\Omega,\beta}^s}^2\leq & C\int_0^T\int_{\Omc}|g^T|^2\;d\mu dt\\
\le &C \left[\int_0^T\|g^T(\cdot,t)\|_{L^2(\Omc,\mu)}^2\;dt + \int_0^T \|u^T(\cdot,t)-u^d\|_{L^2(\Omega)}^2\;dt \right].
\end{align*}
We have shown the claim \eqref{obs1}.

Similarly, using the fact that $g^T=-\psi^T$ in $\Sigma$, we can prove that there is a constant $C>0$ (independent of $T$) such that
\begin{align}\label{obs1-2}
\|\psi^T(\cdot,0)\|_{H_{\Omega,\beta}^s}^2+\int_0^T\mathcal E(\psi^T,\psi^T)\;dt \leq C\int_0^T \|u^T(\cdot,t)-u^d\|_{L^2(\Omega)}^2dt.
\end{align}
Hence, \eqref{obs1-2} implies that
\begin{align}\label{obs2}
\|\psi^T(\cdot,0)\|_{H_{\Omega,\beta}^s}^2\le & C\int_0^T  \|u^T(\cdot,t)-u^d\|_{L^2(\Omega)}^2dt \notag\\
\leq &C\left[\int_0^T  \|u^T(\cdot,t)-u^d\|_{L^2(\Omega)}^2dt +\int_0^T\|g^T(\cdot,t)\|_{L^2(\Omc,\mu)}^2dt \right].
\end{align}

{\bf Step 2:} Since $u^T$ and $\psi^T$ are the weak solutions of \eqref{ocpevol1} and \eqref{ocpevol2}, respectively, it follows from Definition \ref{defweaksol} that
\begin{align}\label{main3}
\langle u_t^T, w\rangle_{(H_{\Omega,\beta}^s)^{*},H_{\Omega,\beta}^s}  +\mathcal E(u^T,w)=\int_{\Omc} g^Tw d\mu, 
\end{align}
for every $ w\in H_{\Omega,\beta}^s$,
and
\begin{align}\label{main4}
-(\psi_t^T, v)_{L^2(\Omega)}  +\mathcal E(\psi^T,v)=\left(u^T-u^d, v\right)_{L^2(\Omega)}, 
\end{align}
for every $ v\in H_{\Omega,\beta}^s$. Taking $v:=u^T$ as a test function in \eqref{main4},  $w:=\psi^T$ as a test function in \eqref{main3} we get that
\begin{align}\label{WW1}
\int_{\Omc}g^T\psi^T\;d\mu=\langle u_t^T, \psi^T\rangle_{(H_{\Omega,\beta}^s)^{*},H_{\Omega,\beta}^s}+(\psi_t^T, u^T)_{L^2(\Omega)}+(u^T-u^d, u^T)_{L^2(\Omega)}.
\end{align}
Integrating \eqref{WW1} over $(0,T)$, using \eqref{main1} and noticing that
\begin{align*}
\int_0^T\Big(\langle u_t^T, \psi^T\rangle_{(H_{\Omega,\beta}^s)^{*},H_{\Omega,\beta}^s}+(\psi_t^T, u^T)_{L^2(\Omega)}\Big)\;dt=0,
\end{align*}
we obtain
\begin{align}\label{m2}
\int_0^T\int_{\Omc}g^T\psi^Td\mu dt&=\int_0^T\Big(u^T(\cdot,t)-u^d, u^T(\cdot,t)\Big)_{L^2(\Omega)} dt\notag\\
&=\int_0^T\|u^T(\cdot,t)\|_{L^2(\Omega)}^2 dt-\int_0^T\Big(u^d, u^T(\cdot,t)\Big)_{L^2(\Omega)}dt.
\end{align}
Using the fact that $g^T=-\psi^T$ in $\Sigma$ and completing the square in the right hand side of \eqref{m2}, we get 
\begin{align}\label{main5}
\int_0^T\|u^T(\cdot,t)-u^d\|_{L^2(\Omega)}^2 dt+\int_0^T\|g^T(\cdot,t)\|_{L^2(\Omc,\mu)}^2dt=-\int_0^T\Big( u^d, u^T(\cdot,t)-u^d\Big)_{L^2(\Omega)}\;dt.
\end{align}
Using the Young inequality in the right hand side of \eqref{main5}, we obtain that for every $\varepsilon>0$,
\begin{align}\label{W4}
-\int_0^T\Big( u^d, u^T(\cdot,t)-u^d\Big)_{L^2(\Omega)}\;dt\leq  \frac{1}{2\varepsilon}T\|u^d\|_{L^2(\Omega)}^2+\frac{\varepsilon}{2}\int_0^T\|u^T(\cdot,t)-u^d\|_{L^2(\Omega)}^2dt.
\end{align}
Choosing $\varepsilon$ in \eqref{W4} such that $1-\frac{\varepsilon}{2}>0$, we can deduce from  \eqref{main5} that
\begin{align}\label{main6}
\int_0^T\|g^T(\cdot,t)\|_{L^2(\Omc,\mu)}^2 dt + \int_0^T\|u^T(\cdot, t)-u^d\|_{L^2(\Omega)}^2 dt \leq CT,
\end{align}
where $C>0$ is a constant depending on $\|u^d\|_{L^2(\Omega)}$, but is independent of $T$.
Combining \eqref{obs1}, \eqref{obs2}, \eqref{main6} and using the fact that
\begin{align*}
\|u\|_{L^2(\Omega)} \le\|u\|_{H_{\Omega,\beta}^s},\;\;\forall\; u\in H_{\Omega,\beta}^s,
\end{align*}
we get the following estimate for $u^T(\cdot, T)$ and $\psi^T(\cdot,0)$:
\begin{align}\label{main7}
\|\psi^T(\cdot,0)\|_{L^2(\Omega)}+\|u^T(\cdot,T)\|_{L^2(\Omega)}\leq C\sqrt{T},
\end{align}
where the constant $C>0$ is independent of $T$.\\

{\bf Step 3:} We claim that
\begin{align*}
\frac{1}{T}\int_0^Tg^T dt \quad \mbox{ and }\quad \frac{1}{T}\int_0^T u^T dt
\end{align*}
are bounded in $L^2(\Omc,\mu)$ and $L^2(\Omega)$, respectively.  

Indeed, it follows from \eqref{main6} and Cauchy--Schwarz's inequality  that
\begin{align*}
\int_{\Omc}\left|\frac{1}{T}\int_0^Tg^T dt\right|^2\;d\mu\le &\int_{\Omc}\frac{1}{T^2}\left(\int_0^T\;dt\right)\left(\int_0^T|g^T|^2\;dt\right)\;d\mu\\
\le& \frac 1T\int_{\Omc}\int_0^T|g^T|^2\;dtd\mu=\frac 1T\int_0^T\int_{\Omc}|g^T|^2\;d\mu dt\le C.
\end{align*}
For $u^T$, we have that
\begin{align*}
\int_{\Om}\left|\frac{1}{T}\int_0^Tu^T dt\right|^2\;dx\le \frac 1T\int_0^T\int_{\Omega}|u^T|^2\;dxdt\le& \frac 1T\int_0^T\int_{\Omega}|u^T-u^d|^2\;dxdt +\frac 1T\int_0^T\int_{\Omega}|u^d|^2\;dxdt\\
\le &C+\|u^d\|_{L^2(\Omega)}^2,
\end{align*}
and we have shown the claim.\\

{\bf Step 4:} Let $w:=u^T-\overline{u}$, $\varphi:=\psi^T-\overline{\psi}$ and $h:=g^T-\overline{g}$. Using Definitions \ref{defweaksol}  we can deduce from \eqref{main1} and \eqref{main2} that
\begin{align}\label{main8}
\begin{cases}
\langle w_t, v\rangle_{(H_{\Omega,\beta}^s)^{*},H_{\Omega,\beta}^s}  +\mathcal E(w,v)=-\displaystyle\int_{\Omc} hv\;d\mu, &\quad \forall v\in H_{\Omega,\beta}^s,\vspace*{0,1cm}\\
-\langle \varphi_t, \phi\rangle_{(H_{\Omega,\beta}^s)^{*},H_{\Omega,\beta}^s}  +\mathcal E(\varphi,\phi)=\displaystyle \int_{\Om}w \phi\;dx, &\quad \forall  \phi\in H_{\Omega,\beta}^s.
\end{cases}
\end{align}
Moreover, $w(\cdot,0)=-\overline{u}$ and $\varphi(\cdot,T)=-\overline{\psi}$.

Now, taking $v:=\varphi$ as a test function in the first equation of \eqref{main8}
and $\phi:=w$ as a test function in the second equation of \eqref{main8} and using the fact that $\varphi:=\psi^T-\overline{\psi}=h$ in $\Sigma$ we get 
\begin{align}\label{M1}
\|w(\cdot,t)\|_{L^2(\Omega)}^2 +\|h(\cdot,t)\|_{L^2(\Omc,\mu)}^2=-\langle \varphi_t, w\rangle_{(H_{\Omega,\beta}^s)^{*},H_{\Omega,\beta}^s} -\langle w_t, \varphi\rangle_{(H_{\Omega,\beta}^s)^{*},H_{\Omega,\beta}^s}.
\end{align}
Integrating \eqref{M1} over $(0,T)$ we obtain
\begin{align}\label{main9}
\int_0^T\|w(\cdot,t)\|_{L^2(\Omega)}^2dt+\int_0^T\|h\|_{L^2(\Omc,\mu)}^2\;dt=-\int_{\Omega} \varphi(x,T) w(x,T)dx+\int_{\Omega} \varphi(x,0)w(x,0)dx.
\end{align}

To get an estimate for the right hand side of \eqref{main9}, we observe the following:
\begin{align*}
\int_{\Omega}|w(x,T)|^2 dx &=\int_{\Omega}|u^T(x,T)-\overline{u}(x)|^2dx\\
&\leq 2\int_{\Omega}|u^T(x,T)|^2dx +2 \int_{\Omega}|\overline{u}(x)|^2 dx
 \leq CT + 2\int_{\Omega}|\overline{u}(x)|^2 dx,
\end{align*}
where  we have used \eqref{main7} in the last estimate.

In a similar way, we have that
\begin{align*}
\int_{\Omega}|\varphi(x,0)|^2 dx \leq CT + 2\int_{\Omega}|\overline{\psi}(x)|^2 dx.
\end{align*}
Thus, the first term in the right hand side of \eqref{main9} can be estimated as follows:
\begin{align*}
-\int_{\Omega} \varphi(x,T) w(x,T)dx &\leq \left(\int_{\Omega}|\varphi(x,T)|^2dx\right)^{\frac 12} \left(\int_{\Omega}|w(x,T)|^2dx\right)^{\frac 12}\\
&\leq \left(\int_{\Omega}|\overline{\psi}(x)|^2dx\right)^{\frac 12} \left(CT +2\int_{\Omega}|\overline{u}(x)|^2\right)^{\frac 12} dx.
\end{align*}

Analogously, for the second term in the right hand side of \eqref{main9} we have that
\begin{align*}
\int_{\Omega} \varphi(x,0)w(x,0)dx \leq \left(CT  +2\int_{\Omega}|\overline{\psi}(x)|^2dx\right)^{\frac 12}\left(\int_{\Omega}|\overline{u}(x)|^2 dx\right)^{\frac 12}.
\end{align*}
We have shown that
\begin{align*}
\frac{1}{T}\int_0^T\|w(\cdot,t)\|_{L^2(\Omega)}^2dt+\frac{1}{T}\int_0^T\|h(\cdot,t)\|_{L^2(\Omc,\mu)}^2 dt \leq & \left(\int_{\Omega}|\overline{\psi}|^2\;dx\right)^{\frac 12}\left(\frac CT+\frac{2}{T^2}\int_{\Omega}|\overline{u}|^2\right)^{\frac 12}\\
&+\left(\int_{\Omega}|\overline{u}|^2\;dx\right)^{\frac 12}\left(\frac CT+\frac{2}{T^2}\int_{\Omega}|\overline{\psi}|^2\right)^{\frac 12}.
\end{align*}

This implies that
\begin{align}\label{m3}
\lim_{T\to\infty} \left(\frac{1}{T}\int_0^T\int_{\Omega}|w|^2dt +\frac{1}{T}\int_0^T\int_{\Omc}|h|^2d\mu dt\right)=0.
\end{align}

{\bf Step 5: } Finally, we show that 
\begin{align*}
\frac{1}{T}\int_0^T g^T dt \quad \mbox{ and }\quad \frac{1}{T}\int_0^T u^T dt,
\end{align*}
converge to $\overline{g}$ and $\overline{u}$ strongly in $L^2(\Omc,\mu)$ and $L^2(\Omega)$, respectively.  

Indeed, we have that
\begin{align*}
\int_{\Omc}\left|\frac{1}{T}\int_0^T g^T dt -\overline{g}\right|^2\;d\mu=&\int_{\Omc}\left|\frac{1}{T}\int_0^T \left(g^T- \overline{g}\right)\; dt\right|^2\;d\mu\\
\le &\frac 1T\int_{\Omc}\int_0^T \left|g^T- \overline{g}\right|^2\;dtd\mu\\
=&\frac{1}{T}\int_0^T\int_{\Omc}|h|^2d\mu dt,
\end{align*}
which converges to $0$ by \eqref{m3}.

Similarly, we have that
\begin{align*}
\int_{\Om}\left|\frac{1}{T}\int_0^T u^T dt -\overline{u}\right|^2\;dx\leq \frac{1}{T}\int_0^T\int_{\Om}|w|^2\;dx dt,
\end{align*}
which converges to $0$ by \eqref{m3}. The proof is finished.
\end{proof}

\begin{remark}
{\em 
We observe that the computations leading to the proof of the estimate \eqref{main6} are used in the proof of our next theorem, otherwise \eqref{main6} can be easily proved as follows. By definition of minimizer, we have that $J_T(g^T)\le J_T(0)$. Hence,
\begin{align*}
\int_0^T\|g^T(\cdot,t)\|_{L^2(\Omc,\mu)}^2 dt &+ \int_0^T\|u^T(\cdot, t)-u^d\|_{L^2(\Omega)}^2 dt\\
&=2J_T(g^T)\le 2J_T(0)= \int_0^T\|0-u^d\|_{L^2(\Omega)}^2 dt=T\|u^d\|_{L^2(\Omega)}^2,
\end{align*}
which is exactly \eqref{main6} with constant $C=1$.
}
\end{remark}

The following exponential turnpike property is our second main result concerning the Robin problem. 

\begin{theorem}\label{mainresult2}
Let $\gamma\geq 0$ be a real number. There is a constant $C=C(\gamma)>0$ (independent of $T$) such that for every $t\in[0,T]$ we have the following estimate:
\begin{align}\label{turnpike}
\|u^T(\cdot,t)-\overline{u}\|_{L^2(\Omega)}+\|\psi^T(\cdot,t)-\overline{\psi}\|_{L^2(\Omega)} \leq C\Big(e^{-\gamma t} + e^{-\gamma(T-t)}\Big) \Big(\|\overline{u}\|_{L^2(\Omega)}+\|\overline{\psi}\|_{L^2(\Omega)}\Big).
\end{align}
\end{theorem}

\begin{proof}
We divide the proof in several steps.\\

{\bf Step 1:} Let $w:=u^T-\overline{u}$ and $\varphi:=\psi^T-\overline{\psi}$ satisfy \eqref{main8}.  It follows from Step 4 in the proof of Theorem \ref{mainresult1} that 
\begin{align}\label{main10}
\begin{cases}
\langle w_t, v\rangle_{(H_{\Omega,\beta}^s)^{*},H_{\Omega,\beta}^s}  +\mathcal E(w,v)=- (\varphi, v)_{L^2(\Omc,\mu)} , &\quad \forall v\in H_{\Omega,\beta}^s,\vspace*{0,1cm}\\
-\langle \varphi_t, v\rangle_{(H_{\Omega,\beta}^s)^{*},H_{\Omega,\beta}^s}  +\mathcal E(\varphi,v)=\displaystyle (w, v)_{L^2(\Omega)} , &\quad \forall  v\in H_{\Omega,\beta}^s.
\end{cases}
\end{align}

Besides, the following identity holds:
\begin{align}\label{main11}
\int_0^T\|w(\cdot,t)\|_{H_{\Omega,\beta}^s}^2dxdt+\int_0^T\|\varphi(\cdot,t)\|_{L^2(\Omc,\mu)}^2 dt=\int_{\Omega} \varphi(x,T) w(x,T)dx-\int_{\Omega} \varphi(x,0)w(x,0)dx.
\end{align}

Next, we use the following notations:
\begin{align}\label{notation1}
(D_t \phi, v):= \langle \phi_t, v\rangle_{(H_{\Omega,\beta}^s)^{*},H_{\Omega,\beta}^s} +\int_{\Omega} \phi(x,0) v dx, \quad \forall v\in H_{\Omega,\beta}^s,
\end{align}
and
\begin{align}\label{notation2}
(D_t^{*} \phi, v):= -\langle \phi_t, v\rangle_{(H_{\Omega,\beta}^s)^{*},H_{\Omega,\beta}^s} +\int_{\Omega} \phi(x,T) v dx, \quad \forall v\in H_{\Omega,\beta}^s.
\end{align}

Using the notations \eqref{notation1}-\eqref{notation2}, we can rewrite the system \eqref{main10} as follows:
\begin{align}\label{main12}
\begin{cases}
(D_t w, v) +\mathcal E(w,v)= - \displaystyle\int_{\Omc} \varphi v d\mu + \int_{\Omega}w(x,0) vdx, &\quad \forall v\in H_{\Omega,\beta}^s,\vspace*{0,1cm}\\
(D_t^{*}\varphi, v) +\mathcal E(\varphi,v)=\displaystyle \int_{\Omega}wv\;dx+\int_{\Omega} \varphi(x,T)v dx,  &\quad \forall v\in H_{\Omega,\beta}^s.
\end{cases}
\end{align}

We observe that we can rewrite \eqref{main12} as follows:
\begin{align}\label{main13}
\left(\begin{array}{cc}
-( \cdot, v)_{L^2(\Omega)} & (D_t^{*}\cdot, v)+\mathcal E(\cdot,v)\\
(D_t\cdot, v)+\mathcal E(\cdot,v) & (\cdot,v)_{L^2(\Omc,\mu)}
\end{array}\right) \left(\begin{array}{c}
w\\
\varphi
\end{array}\right)= \left(\begin{array}{c}
(\varphi(\cdot,T), v)_{L^2(\Omega)}\\
(w(\cdot,0), v)_{L^2(\Omega)}.
\end{array}\right).
\end{align}

Let us denote by $\Lambda$ the matrix operator
\begin{align*}
\Lambda:=\left(\begin{array}{cc}
-(\cdot, v)_{L^2(\Omega)} & (D_t^{*}\cdot, v)+\mathcal E(\cdot,v)\\
(D_t\cdot, v)+\mathcal E(\cdot,v) & (\cdot,v)_{L^2(\Omc,\mu)}
\end{array}\right).
\end{align*}
The operator $\Lambda$ corresponds to the equations for the state and adjoint vectors. The solution operator of this system, denoted $\Lambda^{-1}$, maps initial data for the state and adjoint equations to the corresponding solutions. It follows from Theorem \ref{antil} that the operator $\Lambda^{-1}$ is well defined as a mapping from $(L^2(\Omega))^2$ into $(L^2((0,T);H_{\Omega,\beta}^s)\cap H^1((0,T);(H_{\Omega,\beta}^s)^{*}))^2$. In other words, the state and adjoint equations have unique solutions in $(L^2((0,T);H_{\Omega,\beta}^s)\cap H^1((0,T);(H_{\Omega,\beta}^s)^{*}))^2$. Thus, the operator
$$\Lambda: (L^2((0,T);H_{\Omega,\beta}^s)\cap H^1((0,T);(H_{\Omega,\beta}^s)^{*}))^2 \to (L^2(\Omega))^2$$
 is invertible.\\

{\bf Step 2:} We claim that there is a constant $C>0$ (independent of $T$) such that
\begin{align}\label{normB}
\|\Lambda^{-1}\|_{\mathcal L((L^2(\Omega))^2,(L^2((0,T);H_{\Omega,\beta}^s)\cap H^1((0,T);(H_{\Omega,\beta}^s)^{*}))^2)} \leq C.
\end{align}
Indeed, in order to show the claim we will prove that for every functions $(w,\varphi)$ solving the system \eqref{main12}, there is a constant $C>0$ (independent of $T$) such that
\begin{align}\label{main14}
\|w\|_{L^2((0,T);H_{\Omega,\beta}^s)\cap H^1((0,T);(H_{\Omega,\beta}^s)^{*})}^2+&\|\varphi\|_{L^2((0,T);H_{\Omega,\beta}^s)\cap H^1((0,T);(H_{\Omega,\beta}^s)^{*})}^2\notag\\
&\leq C(\|w(\cdot,0)\|_{L^2(\Omega)}^2+\|\varphi(\cdot,T\|_{L^2(\Omega)}^2).
\end{align}

In fact, proceeding as  the proof of \eqref{obs1}, we obtain that there is a constant $C>0$ (independent of $T$) such that
\begin{align}\label{main15}
\begin{cases}
\|w(\cdot,T)\|_{L^2(\Omega)}^2\leq C \left[\displaystyle\int_0^T\|\varphi(\cdot,t)\|_{L^2(\Omc,\mu)}^2\;dt +\|w(\cdot,0)\|_{L^2(\Omega)}^2\right],\vspace*{0.1cm}\\
\|\varphi(\cdot,0)\|_{L^2(\Omega)}^2\leq C \left[\displaystyle\int_0^T\|w(\cdot,t)\|_{L^2(\Omega)}^2 dt +\|\varphi(\cdot,T)\|_{L^2(\Omega)}^2 \right]. 
\end{cases}
\end{align}

Using the estimates \eqref{obs1-1} and \eqref{obs1-2}, we get that
\begin{align}\label{main16}
\begin{cases}
\|w\|_{L^2((0,T);H_{\Omega,\beta}^s)}^2 \leq C \displaystyle\left[\int_0^T\|\varphi(\cdot,t)\|_{L^2(\Omc,\mu)}^2\;dt+ \int_0^T\|w(\cdot,t)\|_{L^2(\Omega)}^2dt +\|w(\cdot,0)\|_{L^2(\Omega)}^2\right],\vspace*{0.1cm}\\
\|\varphi\|_{L^2((0,T);H_{\Omega,\beta}^s)}^2\leq C\displaystyle\left[\int_0^T\|\varphi(\cdot,t)\|_{L^2(\Omc,\mu)}^2\;dt+ \int_0^T \|w(\cdot,t)\|_{L^2(\Omega)}^2dt+ \|\varphi(\cdot,T)\|_{L^2(\Omega)}^2\right].
\end{cases}
\end{align}

Using the identity \eqref{main11}, Young's inequality, the fact that $\|u\|_{H_{\Omega,\beta}^s}\ge\|u\|_{L^2(\Omega)}$, $u\in H_{\Omega,\beta}^s$, and  \eqref{main15} we get that for every $\varepsilon_1$, $\varepsilon_2>0$,
\begin{align}\label{jj}
&\int_0^T\|w(\cdot,t)\|_{H_{\Omega,\beta}^s}^2dt+\int_0^T\|\varphi(\cdot,t)\|_{L^2(\Omc,\mu)}^2\;dt\notag\\
\leq &C\left[ \|\varphi(\cdot,T)\|_{L^2(\Omega)}^2\left(\frac{1}{2\varepsilon_1}+\frac{1}{2\varepsilon_2}\right) +\|w(\cdot,0)\|_{L^2(\Omega)}^2\frac{(\varepsilon_1+\varepsilon_2)}{2} \right.\notag\\ 
&+ \left.  \frac{\varepsilon_1}{2}\int_0^T\|\varphi(\cdot,t)\|_{L^2(\Omc,\mu)}^2\;dt+\frac{1}{2\varepsilon_2}\int_0^T\|w(\cdot,t)\|_{L^2(\Omega)}^2dt \right].
\end{align}
Noticing that $\|w(\cdot,t)\|_{L^2(\Omega)}\le\|w(\cdot,t)\|_{H_{\Omega,\beta}^s}$ and choosing $\varepsilon_1,\varepsilon_2>0$ such that the last two terms in the right hand side of \eqref{jj} can be absorbed by the left hand side, we obtain that there is a constant $C>0$ (independent of $T$) such that
\begin{align}\label{main17}
\int_0^T\|w(\cdot,t)\|_{H_{\Omega,\beta}^s}^2dt+\int_0^T\|\varphi(\cdot,t)\|_{L^2(\Omc,\mu)}^2\;dt\leq C \Big(\|w(\cdot,0)\|_{L^2(\Omega)}^2+ \|\varphi(\cdot,T)\|_{L^2(\Omega)}^2\Big).
\end{align}
Combining \eqref{main16}-\eqref{main17}, we can deduce that there is a constant $C>0$ (independent of $T$) such that
\begin{align}\label{main18}
\|w\|_{L^2((0,T);H_{\Omega,\beta}^s)}^2 +\|\varphi\|_{L^2((0,T);H_{\Omega,\beta}^s)}^2\leq C \Big(\|w(\cdot,0)\|_{L^2(\Omega)}^2+ \|\varphi(\cdot,T)\|_{L^2(\Omega)}^2\Big).
\end{align}

Since $w$ satisfies the first equation in \eqref{main10}, using the continuity of $\mathcal E$, we have that
\begin{align*}
\|w_t\|_{L^2((0,T);(H_{\Omega,\beta}^s)^{*})}\leq C(\|w\|_{L^2((0,T);H_{\Omega,\beta}^s)} + \|\varphi\|_{L^2((0,T);L^2(\Omc,\mu))}).
\end{align*}
In a similar way, we can obtain an estimate for $\varphi_t$ and  the claim follows.\\

{\bf Step 3:} Define 
\begin{align*}
\begin{cases}
\widetilde{w}&:=\displaystyle\frac{1}{e^{-\gamma t}+ e^{-\gamma(T-t)}}w,\vspace{0.1cm}\\
\widetilde{\varphi}&:=\displaystyle\frac{1}{e^{-\gamma t}+ e^{-\gamma(T-t)}}\varphi,
\end{cases}
\end{align*}
where we recall that $w:=u^T-\overline{u}$ and $\varphi:=\psi^T-\overline{\psi}$.  Calculating we get that $\widetilde{w}$ and $\widetilde{\varphi}$ solve the following problems:
\begin{align*}
\begin{cases}
\widetilde{w}_t+(-\Delta)^s\widetilde{w} =\gamma f(t)\widetilde{w} & \text{ in }Q,\\
\mathcal{N}_s \widetilde{w}+  \beta\widetilde{w} = -\beta\widetilde{\varphi} & \text{ in }\Sigma,\\
\widetilde{w}(\cdot, 0) =\displaystyle\frac{-1}{1+e^{-\gamma T}}\overline{u} & \text{ in }\Omega,
\end{cases}
\end{align*}
and
\begin{align*}
\begin{cases}
-\widetilde{\varphi}_t+(-\Delta)^s \widetilde{\varphi}=\widetilde{w} -\gamma f(t)\widetilde{\varphi} & \text{ in }Q,\\
\mathcal{N}_s \widetilde{\varphi} + \beta\widetilde{\varphi} = 0 & \text{ in }\Sigma,\\
\widetilde{\varphi}(\cdot,T) =\displaystyle\frac{-1}{1+e^{-\gamma T}}\overline{\psi} & \text{ in }\Omega.
\end{cases}
\end{align*}
Here, $f$ denotes the time-dependent function given by
\begin{align}\label{ec3}
f(t):=\displaystyle\frac{e^{-\gamma(T-t)}-e^{-\gamma t}}{e^{-\gamma t} + e^{-\gamma (T-t)}}.
\end{align}

Using Definition \ref{defweaksol}, we obtain
\begin{align}\label{ec1}
\langle \widetilde{w}_t, v\rangle_{(H_{\Omega,\beta}^s)^{*},H_{\Omega,\beta}^s}  +\mathcal E(\widetilde{w},v)=-\displaystyle\int_{\Omc} \widetilde{\varphi}v\; d\mu +\gamma f(t)\int_{\Omega}\widetilde{w}v\;dx,
\end{align}
and
\begin{align}\label{ec2}
-\langle \widetilde{\varphi}_t, v\rangle_{(H_{\Omega,\beta}^s)^{*},H_{\Omega,\beta}^s}  +\mathcal E(\widetilde{\varphi},v)=\displaystyle \int_{\Om} \widetilde{w}v\;dx -\gamma f(t)\int_{\Omega}\widetilde{\varphi} v\;dx,
\end{align}
for every $v\in H_{\Omega,\beta}^s$. 

Using \eqref{notation1} and \eqref{notation2}, we can rewrite  \eqref{ec1} and \eqref{ec2} as follows:
\begin{align}\label{ec4}
(D_t \widetilde{w}, v) +\mathcal E(\widetilde{w},v)=-\displaystyle\int_{\Omc} \widetilde{\varphi}v d\mu +\gamma f(t)\int_{\Omega}\widetilde{w}vdx + \int_{\Omega}\widetilde{w}(\cdot,0) vdx ,
\end{align}
and
\begin{align}\label{ec5}
(D_t^{*}\widetilde{\varphi}, v) +\mathcal E(\widetilde{\varphi},v)=\displaystyle \int_{\Omega} \widetilde{w}v\;dx-\gamma f(t)\int_{\Omega}\widetilde{\varphi} v\;dx +\int_{\Omega} \widetilde{\varphi}(\cdot,T) v\; dx.
\end{align}

Following Step 1,  we can rewrite \eqref{ec4} and \eqref{ec5} as follows:
\begin{multline}\label{ec6}
\left[\left(\begin{array}{cc}
-(\cdot, v)_{L^2(\Omega)} & (D_t^{*}\cdot, v)+\mathcal E(\cdot,v)\\
(D_t\cdot, v)+\mathcal E(\cdot,v) & (\cdot,v)_{L^2(\Omc,\mu)}
\end{array}\right) -\gamma\left(\begin{array}{cc}
0 & -f(t)(\cdot,v)_{L^2(\Omega)}\\
f(t)(\cdot,v)_{L^2(\Omega)} & 0
\end{array}\right)\right] \left(\begin{array}{c}
\widetilde{w}\\
\widetilde{\varphi}
\end{array}\right)\\= \left(\begin{array}{c}
(\widetilde{\varphi}(\cdot,T), v)_{L^2(\Omega)}\\
(\widetilde{w}(\cdot,0), v)_{L^2(\Omega)}.
\end{array}\right).
\end{multline}

Let us denote by $\mathbb{F}$ the operator matrix
\begin{align*}
\mathbb{F}:=\left(\begin{array}{cc}
0 & -f(t)(\cdot,v)_{L^2(\Omega)}\\
f(t)(\cdot,v)_{L^2(\Omega)} & 0
\end{array}\right).
\end{align*}

We claim that $\|\mathbb{F}\|_{\mathcal L((L^2(\Om\times(0,T))^2)}\leq 1$. Indeed,  for every $\xi\in L^2(\Om\times (0,T))$ and $v\in L^2(\Om)$ we have that
\begin{align*}
\left|\int_0^T f(t)(\xi(\cdot,t),v)_{L^2(\Omega)}dt\right|&\leq \int_0^T\displaystyle\frac{e^{-\gamma(T-t)}-e^{-\gamma t}}{e^{-\gamma t} + e^{-\gamma (T-t)}}\Big|(\xi(\cdot,t),v)_{L^2(\Omega)}\Big|dt\\
&\leq \int_0^T\Big|(\xi(\cdot,t),v)_{L^2(\Omega)}\Big|dt\leq \|\xi\|_{L^2(\Om\times(0,T))}\|v\|_{L^2(\Omega)},
\end{align*}
and the claim is proved.

Let us choose $\gamma>0$ such that
\begin{align}\label{gamma}
\alpha:=\gamma\|\Lambda^{-1}\|_{\mathcal L((L^2(\Omega))^2, (L^2((0,T);H_{\Omega,\beta}^s)\cap H^1((0,T);(H_{\Omega,\beta}^s)^{*}))^2)}<1.
\end{align}
Since the norm of $\Lambda^{-1}$ is independent of $T$ (see \eqref{normB}), it follows that $\gamma$ is also independent of $T$.\\

{\bf Step 4:}
Let $\widetilde{\Phi}:=(\widetilde{w},\widetilde{\varphi})^T$.  Then, the system \eqref{ec6} is equivalent to
\begin{align}\label{ec7}
(\Lambda-\gamma \mathbb{F})\widetilde{\Phi}=\mathcal{I}, 
\end{align}
where we have set
\begin{align*}
\mathcal{I}:=\left(\begin{array}{c}
(\widetilde{\varphi}(\cdot,T), v)_{L^2(\Omega)}\\
(\widetilde{w}(\cdot,0), v)_{L^2(\Omega)}
\end{array}\right).
\end{align*}
Since $\Lambda$ is invertible,  it follows that \eqref{ec7} is equivalent to
\begin{align}\label{ec8}
(I-\gamma \Lambda^{-1}\mathbb{F})\widetilde{\Phi}=\Lambda^{-1}\mathcal{I}.
\end{align}
Using  \cite[Theorem 2.14]{kress1989linear}, the existence and uniqueness of solutions for operator equations as \eqref{ec7} can be established in terms of Neumann series, provided that $\gamma \Lambda^{-1}\mathbb{F}$ is a contraction. Since we have chosen $\gamma$ such that 
\begin{align*}
\gamma\|\Lambda^{-1}\|_{\mathcal L((L^2(\Omega))^2,(L^2((0,T);H_{\Omega,\beta}^s)\cap H^1((0,T);(H_{\Omega,\beta}^s)^{*}))^2)}<1\;\mbox{ and }\; \|\mathbb{F}\|_{L^2((0,T);L^2(\Omega))}\leq 1,
\end{align*} 
it follows that $I-\gamma \Lambda^{-1}\mathbb{F}$ has a bounded inverse which is given by the following Neumann series:
\begin{align*}
(I-\gamma \Lambda^{-1}\mathbb{F})^{-1}=\displaystyle\sum_{k=0}^{\infty}\Big(\gamma \Lambda^{-1}\mathbb{F}\Big)^k.
\end{align*}
In addition, we have that
\begin{align}
\|I-\gamma \Lambda^{-1}\mathbb{F}\|_{\mathcal L((L^2((0,T);H_{\Omega,\beta})\cap H^1((0,T);(H_{\Omega,\beta}^s)^{*}))^2)}\leq\displaystyle\frac{1}{1-\alpha},
\end{align}
where we recall that $0<\alpha<1$ is given in \eqref{gamma}.

Therefore, we can deduce that
\begin{align}\label{final-turnpike}
&\left\|\frac{1}{e^{-\gamma t}+ e^{-\gamma(T-t)}} w\right\|_{L^2((0,T);H_{\Omega,\beta}^s)\cap H^1(0,T;(H_{\Omega,\beta}^s)^{*})}+ \left\|\frac{1}{e^{-\gamma t}+ e^{-\gamma(T-t)}}\varphi\right\|_{L^2((0,T);H_{\Omega,\beta}^s)\cap H^1((0,T);(H_{\Omega,\beta}^s)^{*})}\notag \\ 
\leq &\frac{\|\Lambda^{-1}\|_{(L^2(\Omega))^2}}{1-\alpha }\frac{1}{1+e^{-\gamma T}}\Big(\|w(\cdot,0)\|_{L^2(\Omega)}+ \|\varphi(\cdot,T)\|_{L^2(\Omega)}\Big).
\end{align}

Since $H_{\Omega,\beta}^s$ is a Hilbert space, using 
\cite[Proposition 23.23]{zeidler1990nonlinear}, we have that the continuous embedding
\begin{align*}
L^2((0,T);H_{\Omega,\beta}^s)\cap H^1((0,T);(H_{\Omega,\beta}^s)^{*})\hookrightarrow C([0,T];L^2(\Omega)),
\end{align*}
 holds. In addition, the embedding constant is independent of $T$. 

Thus, we can deduce from \eqref{final-turnpike} that for every $t\in [0,T]$, the following estimate holds:
\begin{align*}
\|u^T(\cdot,t)-\overline{u}\|_{L^2(\Omega)}+\|\psi^T(\cdot,t)-\overline{\psi}\|_{L^2(\Omega)}\leq C\Big(e^{-\gamma t}+e^{-\gamma(T-t)}\Big)\Big(\|\overline{u}\|_{L^2(\Omega)}+\|\overline{\psi}\|_{L^2(\Omega)}\Big),
\end{align*}
with a constant $C>0$ which is independent of $T$.
We have shown \eqref{turnpike} and the proof is finished.
\end{proof}

We conclude this section with the following observation.

\begin{remark}\label{rem-310}
{\em We observe the following facts.
\begin{enumerate}
\item From our previous results, we can also obtain an estimate for the control. Indeed, let 
\begin{align*}
\widetilde{h}:=\frac{1}{e^{-\gamma t}+e^{-\gamma(T-t)}}(g^T-\overline{g}).
\end{align*}
Since $\|u\|_{L^2(\Omc,\mu)}\leq \|u\|_{H_{\Omega,\beta}^s}$ for $u\in H_{\Omega,\beta}^s$,  we have that (notice that $\varphi:=\psi^T-\overline{\psi}=g^T-\overline{g}$ in $\Sigma$)
\begin{align*}
\|\widetilde{h}\|_{L^2((0,T);L^2(\Omc,\mu))}^2=\int_0^T \left(\frac{1}{e^{-\gamma t}+e^{-\gamma(T-t)}}\right)^2\|\varphi(\cdot,t)\|_{L^2(\Omc,\mu)}^2 dt
\leq \|\widetilde{\varphi}\|_{L^2((0,T);H_{\Omega,\beta}^s)}^2,
\end{align*}
where 
\begin{align*}
\widetilde{\varphi}:=\frac{1}{e^{-\gamma t}+e^{-\gamma(T-t)}}\varphi.
\end{align*}
Thus, from \eqref{final-turnpike}, we can deduce that there is a constant $C>0$ (independent on $T$)  such that
\begin{align*}
\left\|\frac{1}{e^{-\gamma t}+e^{-\gamma(T-t)}}(g^T-\overline{g})\right\|_{L^2((0,T);L^2(\Omc,\mu))}\leq C\Big(\|\overline{u}\|_{L^2(\Omega)}+ \|\overline{\psi}\|_{L^2(\Omega)}\Big).
\end{align*}
We have shown the exponential turnpike property (state, control and adjoint vectors) for the Robin control problems.
\item We do not know if the estimate \eqref{turnpike} can be improved as follows:
\begin{align}\label{T-P}
\|u^T(\cdot,t)-\overline{u}\|_{L^2(\Omega)}+\|\psi^T(\cdot,t)-\overline{\psi}\|_{L^2(\Omega)} \leq C \Big(e^{-\gamma t} \|\overline{u}\|_{L^2(\Omega)} +e^{-\gamma(T-t)}\|\overline{\psi}\|_{L^2(\Omega)} \Big).
\end{align}
Such an improved estimate has been obtained in \cite{porretta2016remarks} for the local case $s=1$, with the control function localized in $\Omega$ and zero boundary conditions, by using Riccati's theory for infinite dimensional systems. It seems that our method cannot be used to obtain \eqref{T-P}. To obtain such an estimate, most likely, one has to generalize the Riccati theory to the fractional setting and also exploit some abstract results contained in \cite[Chapter III]{lions1971optimal} about decoupling the optimality systems.
\end{enumerate}
}
\end{remark}

\section{Dirichlet exterior control problems: The turnpike property}\label{sec-6}

In this section we prove the exponential turnpike property for the Dirichlet exterior control problem. In order to do this, we need some preparations. Recall that $(-\Delta)_D^s$ denotes the operator defined in \eqref{DLO}.

\subsection{The Dirichlet exterior control problem}\label{sec-5}

In this section we give some known results needed to formulate our problem. These results will also be used in the proofs of the turnpike property. 

Let us consider first the optimal control for the stationary problem. That is, 
\begin{align}\label{FunctionalStationary}
\min_{g\in L^2(\Omc)} J(g):=\frac{1}{2}\|u-u^d\|_{L^2(\Omega)}^2 + \frac{1}{2}\|g\|_{L^2(\Omc)}^2,
\end{align}
subject to $u$ solves the following elliptic problem:
\begin{align}\label{Dir-Eq}
\begin{cases}
(-\Delta)^s u =0 & \text{ in }\Omega,\\
u= g & \text{ in }\Omc,
\end{cases}
\end{align}
where $u^d\in L^2(\Omega)$ is a fixed target.

Our notion of solution to \eqref{Dir-Eq} is the following.

\begin{definition}\label{def-weak-Dir}
Let $g\in L^2(\Omc)$. We say that $u\in L^2(\Omega)$ is a solution by transposition (or very weak solution) of \eqref{Dir-Eq} if the identity
\begin{align}\label{transposition}
\int_{\Omega}u(-\Delta)v\;dx= -\int_{\Omc}g\mathcal{N}_s v\; dx,
\end{align}
holds for every $v\in D((-\Delta)_D^s)=\Big\{v\in {H}_0^s(\Omega): \ (-\Delta)^s v\in L^2(\Omega)\Big\}$.
\end{definition}

The following existence and uniqueness result of solutions by transposition has been recently obtained in \cite[Theorem 3.5]{antil2019external}. 

\begin{theorem}\label{exis-uni-trans}
Let $g\in L^2(\Omc)$. There exists a unique solution by transposition $u$ to \eqref{Dir-Eq} in the sense of Definition \ref{def-weak-Dir}. In addition, there is a constant $C>0$ such that
\begin{align}
\|u\|_{L^2(\Omega)} \leq C \|g\|_{L^2(\Omc)}.
\end{align}
\end{theorem}

With respect to the Dirichlet optimal control problem, we also have the following result taken from \cite[Theorems 4.1 and 4.3]{antil2019external}.

\begin{theorem}\label{minimumstationary}
There exists a solution $(\overline{g},\overline{u})\in L^2(\Omc)\times L^2(\Omega)$ to the minimization problem \eqref{FunctionalStationary}-\eqref{Dir-Eq}. In addition, the first order necessary optimality conditions are given by
\begin{align}
\overline{g}=\mathcal{N}_s \overline{\lambda},
\end{align}
where $\overline{\lambda}\in D((-\Delta)_D^s) \hookrightarrow{H}_0^s(\Omega)$ solves the following adjoint equation:
\begin{align}\label{stationary-Dir}
\begin{cases}
(-\Delta)^s \overline{\lambda} =\overline{u}-u^d & \text{ in }\Omega,\\
\overline{\lambda}= 0 & \text{ in }\Omc.
\end{cases}
\end{align}
\end{theorem}

Now, we consider the evolutionary optimal control problem. That is,
\begin{align}\label{evol-func}
\min_{g\in L^2((0,T)\times ((\Omc))} J^T(g):=\frac{1}{2}\int_0^T\|u(\cdot,t)-u^d\|_{L^2(\Om)}^2\;dt + \frac{1}{2}\int_0^T\|g(\cdot,t)\|_{L^2(\Omc)}^2 \;dt,
\end{align}
subject to $u$ solves the following fractional heat equation with a Dirichlet exterior datum:
\begin{align}\label{evol-1}
\begin{cases}
u_t+(-\Delta)^s u=0 & \text{ in } Q\\
u =g & \text{ in }\Sigma\\
u(\cdot, 0) =0 & \text{ in }\Omega.
\end{cases}
\end{align}
Once again, $u^d\in L^2(\Omega)$ is a fixed target.

We introduce our notion of solutions.

\begin{definition}
Let $g\in L^2((0,T);L^2(\Omc))$. A function $u\in L^2(\Omega\times (0,T))$ is said to be a solution by transposition (or very weak solutions) of \eqref{evol-1}, if the identity
\begin{align}
\int_0^T\int_{\Omega}u\Big(-v_t+(-\Delta)^s v\Big) dxdt= -\int_0^T\int_{\Omc}g\mathcal{N}_s v dxdt,
\end{align}
holds, for every $v\in L^2((0,T); D((-\Delta)_D^s))\cap H^1((0,T); L^2(\Omega))$ with $v(\cdot,T)=0$.
\end{definition}

For the optimal control problem \eqref{evol-func}-\eqref{evol-1}, the following existence result has been obtained in \cite[Theorems 4.1 and 4.3]{antil2020external}.

\begin{theorem}
There exists an optimal pair $(g^T,u^T)$ to the minimization problem \eqref{evol-func}-\eqref{evol-1}. Moreover, the first order optimality conditions are given by
\begin{align}
g^T=\mathcal{N}_s \lambda^T,
\end{align}
where $\lambda^T\in L^2((0,T);D((-\Delta)_D^s)\cap H^1((0,T); L^2(\Omega))$ solves the following adjoint problem:
\begin{align}\label{evol-2}
\begin{cases}
-\lambda_t^T+(-\Delta)^s \lambda^T=u^T-u^d & \text{ in } Q\\
\lambda^T =0 & \text{ in } \Sigma,\\
\lambda^T(\cdot, T) =0 & \text{ in }\Omega.
\end{cases}
\end{align}
\end{theorem}

Let $(e^{-t(-\Delta)_D^s})_{t\geq 0}$ be the submarkovian  semigroup on $L^2(\Omega)$ generated by the operator $-(-\Delta)^s_D$. Then, the solution $\lambda^T$ of \eqref{evol-2} is given by
\begin{align*}
\lambda^T(\cdot,t)=\int_t^Te^{-(\tau-t)(-\Delta)_D^s}\Big(u^T(\cdot,\tau)-u^d\Big)\;d\tau.
\end{align*}

\subsection{Admissible control and observation operators}\label{sec42}
To prove the turnpike property in the case of the Dirichlet exterior control, we will use the approach of \emph{admissible control and observation operators} that we explain next. As we have mentioned in the introduction, these concepts shall allow us to use semigroups theory to prove existence of solutions to \eqref{evol-1} and to obtain some regularity results. More precisely, we will obtain a continuous in time optimal state, which is crucial to get the exponential turnpike property.

Let us denote by  $\mathbb{A}:=(-\Delta)_D^s$, with $D(\mathbb{A}):=D((-\Delta)_D^s)$.  That is, $\mathbb A$ is the realization in $L^2(\Omega)$ of the fractional Laplace operator with zero Dirichlet exterior condition defined in \eqref{DLO}. The operator $\mathbb A$ can be extended to a bounded operator from $H_0^s(\Omega)$ into $H^{-s}(\Omega)$. If there is no confusion we use the same notation $\mathbb A$. Then, the operator
$-\mathbb{A}: D(\mathbb{A})\subset {H}_0^{s}(\Omega)\to {H}^{-s}(\Omega)$ generates a strongly continuous submarkovian semigroup $(\mathbb{T}(t))_{t\geq 0}$ on ${H}^{-s}(\Omega)$ which coincides with the semigroup  $(e^{-t(-\Delta)_D^s})_{t\geq 0}$ on $L^2(\Omega)$.

Let $(D(\mathbb A))^\star$ denote the dual space of $D(\mathbb A)$ so that we have the following continuous and dense embeddings:
\begin{align*}
D(\mathbb A)\hookrightarrow {H}^{-s}(\Omega) \hookrightarrow (D(\mathbb A))^\star.
\end{align*}
The semigroup $\mathbb{T}$ can be also extended to $(D(\mathbb A))^\star$ and its generator is an extension of $\mathbb A$. We notice that the semigroup $\mathbb{T}$ is exponentially stable.
We refer to the book of Tucsnak and Weiss \cite[Chapter 2]{tucsnak2009observation} for an abstract version and further properties.

 Throughout the following, if there is no confusion, we shall only denote by  $\mathbb T$ any of the above mentioned three semigroups.

The following definitions are inspired from \cite[Section 4.2 and 4.3]{tucsnak2009observation}. 

\begin{definition}\label{def-admissible}
\begin{enumerate}
\item An operator $\mathbb{B}\in \mathcal{L}(L^2(\Omc);(D(\mathbb A))^\star)$ is called an admissible control operator for the semigroup $(\mathbb{T}(t))_{t\geq 0}$, if for some $\tau>0$, $\operatorname{Rang}(\Phi_{\tau})\subset {H}^{-s}(\Omega)$, where for $g\in L^2((0,T); L^2(\Omc))$ we have set
\begin{align*}
\Phi_{\tau}g(t):=\int_0^{t}\mathbb{T}(t-\tau)\mathbb{B}g(\tau)\;d\tau.
\end{align*}
\item An operator $\mathbb{E}\in \mathcal{L}(D(\mathbb A),L^2(\Omega))$ is called an admissible observation operator for the semigroup $(\mathbb{T}(t))_{t\geq 0}$, if for some $\tau>0$, $\Psi_{\tau}$ has a continuous extension to ${H}^{-s}(\Omega)$, where for $u_0\in D(\mathbb A)$,
\begin{align*}
(\Psi_{\tau}u_0)(t):=
\begin{cases}
\mathbb{E}\mathbb{T}(t)u_0 & \text{ if } t\in [0,\tau],\\
0 & \text{ if } t>\tau.
\end{cases}
\end{align*}
\end{enumerate}
\end{definition}


\begin{remark}
{\em We observe the following.
\begin{enumerate}
\item  An admissible control operator $\mathbb{B}$ is called bounded if $\mathbb{B}\in \mathcal{L}(L^2(\Omc); {H}^{-s}(\Omega))$, and unbounded otherwise. Obviously, every bounded operator $\mathbb{B}\in \mathcal{L}(L^2(\Omc); {H}^{-s}(\Omega))$ is an admissible control operator  for $\mathbb{T}$. 

\item An admissible observation operator $\mathbb{E}$ is called bounded if it can be extended such that $\mathbb{E}\in \mathcal{L}({H}^{-s}(\Omega); L^2(\Omega))$, and unbounded otherwise. Once again, every bounded linear operator $\mathbb{E}\in \mathcal{L}({H}^{-s}(\Omega); L^2(\Omega))$ is an admissible observation operator for the semigroup $\mathbb{T}$.
\end{enumerate}
}
\end{remark}

With the previous notations, we consider the control operator $\mathbb{B}\in \mathcal{L}(L^2(\Omc); (D(\mathbb A))^\star)$ defined by $\mathbb{B}:=\mathbb{A}\mathbb{D}$, where $\mathbb{D}:L^2(\Omc)\to L^2(\Omega)$ is the nonlocal Dirichlet map given by 
\begin{align}\label{Dir-map}
\mathbb{D}g=u \Longleftrightarrow (-\Delta)^s u=0 \ \text{in }\Omega \ \text{and }u=g \ \text{in }\Omc.
\end{align}
It follows from Theorem \ref{exis-uni-trans} that, for every $g\in L^2(\Omc)$, there exists a unique function $u\in L^2(\Omega)$ satisfying \eqref{Dir-map}.  Therefore, the nonlocal Dirichlet exterior control problem \eqref{evol-1} can be rewritten as follows:
\begin{align}\label{evolution-new}
\begin{cases}
u_t+\mathbb{A}u=\mathbb{B}g, & \ t>0,\\
u(0)=0, &
\end{cases}
\end{align}
and $\mathbb{B}$ is an admissible control operator for the semigroup $\mathbb{T}$ generates by $-\mathbb{A}$, in the sense of Definition \ref{def-admissible}. Notice that \eqref{evolution-new} has a unique solution $u\in L^2(\Omega\times (0,T))\cap C([0,T]; H^{-s}(\Omega))$ given by 
\begin{align}
u(\cdot,t)=\int_0^t \mathbb{T}(t-\tau)\mathbb{B}g(\cdot, \tau)\;d\tau.
\end{align}

Moreover, the operator $\mathbb{B}^{*}:\mathcal{L}(D(\mathbb A); L^2(\Omc))$ is given by
\begin{align}
\mathbb{B}^{*}\varphi =-\mathcal{N}_s(\mathbb A^{-1}\varphi), \quad \forall \varphi\in L^2(\Omega).
\end{align}
Finally, it follows from \cite[Theorem 4.4.3]{tucsnak2009observation}  that $\mathbb{B}^{*}$ is an admissible observation operator for the semigroup $(\mathbb{T}(t))_{t\geq0}$ and, since $(\mathbb{T}(t))_{t\geq0}$ is exponentially stable, we have that
\begin{align}\label{admissible-obs}
\int_0^t\|\mathbb{B}^{*}\mathbb{T}(t-\tau)u(\tau)\|_{L^2(\Omc)}^2 d\tau \leq K\|u(\cdot,t)\|_{H^{-s}(\Omega)}^2,
\end{align}
where the  constant $K>0$ can be chosen independent of $t$ (see e.g. \cite[Proposition 4.4.3 and Remark 4.3.5]{tucsnak2009observation}).

To conclude this section,  we observe that, since the operator $-\mathbb{A}$ generates a strongly continuous semigroup $(\mathbb{T}(t))_{t\geq 0}$ on ${H}^{-s}(\Omega)$ (or on $(D(\mathbb A))^\star$), and $\mathbb{B}\in \mathcal{L}(L^2(\Omc); (D(\mathbb A))^\star)$ is an admissible control operator, considering the observation operator $\mathbb{E}:=I$, it follows from \cite[Proposition 4.9]{tucsnak2014well} that the triple $(\mathbb{A},\mathbb{B},\mathbb{E})$ forms a  well--posed system in the sense of \cite[Definition 3.1]{tucsnak2014well}. Therefore, the extremal system associated to the optimal control problem \eqref{evol-func}, which can be rewritten as follows:
\begin{align}\label{Dir-extremal-evol}
\left(\begin{array}{cc}
-I & -\frac{d}{dt}+\mathbb A\vspace*{0.1cm}\\
0 & E_T\vspace*{0.1cm}\\
\frac{d}{dt}+\mathbb A & \mathbb{B}\mathbb{B}^{*}\vspace*{0.1cm}\\
E_0 & 0
\end{array}\right)\left(\begin{array}{c}
u^{T}\\
\lambda^T
\end{array}\right)=\left(\begin{array}{c}
-u^d\vspace*{0.1cm}\\
0\vspace*{0.1cm}\\
0\vspace*{0.1cm}\\
0
\end{array}\right),
\end{align}
where $E_0u^T:=u(\cdot,0)$ and $E_T\lambda^T:=\lambda^T(\cdot,T)$, 
admits  a unique solution $(u^T,\lambda^T)\in C([0,T]; {H}^{-s}(\Omega))\times C([0,T];D((-\Delta)_D^s)$.

Finally, we observe that we can rewrite the extremal equation for the stationary optimal control problem \eqref{stationary-Dir} as follows:
\begin{align}\label{Dir-extremal-sta}
\left(\begin{array}{cc}
-I & \mathbb A\vspace*{0.1cm}\vspace*{0.1cm}\\
\mathbb A & \mathbb{B}\mathbb{B}^{*}
\end{array}\right)\left(\begin{array}{c}
\overline{u}\vspace*{0.1cm}\\
\overline{\lambda}
\end{array}\right)=\left(\begin{array}{c}
-u^d\vspace*{0.1cm}\\
0
\end{array}\right),
\end{align}
where $\mathbb B^\star\overline{\lambda}=-\overline{g}$ in $\Omc$.

\subsection{The turnpike property}\label{subsec-6}

Before we state and prove our last main result, we need the following technical results.

The first one is an integration by parts formula for mild solutions of abstract Cauchy problems.

\begin{proposition}{\cite[Chapter 2, Proposition 5.7]{li1995optimal}}
Let $\mathcal H$ be a Hilbert space and $A:D(A)\to \mathcal H$ the generator of a strongly continuous semigroup $(\mathcal{T}(t))_{t\geq 0}$ on $\mathcal H$. Let $X_0, Y_0\in \mathcal H$, $F,G\in L^2((0,T);\mathcal H)$ and consider the following nonhomogeneous linear Cauchy problems:
\begin{align*}
\frac{dX}{dt}-AX=F, \quad X(0)=X_0,\\
-\frac{dY}{dt}-A^{*}Y=G, \quad Y(0)=Y_0.
\end{align*}
Then, for every $0\leq \tau\leq t$, the following integration by parts formula holds: 
\begin{align}\label{integration-by-parts}
( X(t),Y(t))_{\mathcal{H}} - ( X(\tau),Y(\tau))_{\mathcal{H}} = \int_{\tau}^t \Big(( Y(\sigma),F(\sigma))_{\mathcal{H}} - ( X(\sigma),G(\sigma))_{\mathcal{H}} \Big) d\sigma.
\end{align}
\end{proposition}

\begin{proposition}{\cite[Lemma 7]{grune2020exponential}}\label{prop-grune}
For $\eta\in L^1((0,T))$, we define the function
\begin{align*}
h(t):=\int_0^t \eta(\tau)e^{-k(t-\tau)}\;d\tau,
\end{align*}
where $k>0$. Then, there is a constant $C>0$ (independent of $T$) such that
\begin{align}
\|h\|_{L^p(0,T)}\leq C\|\eta\|_{L^1(0,T)}, \quad \text{ for }1\leq p\leq \infty.
\end{align}
\end{proposition}

Our third main result, which is the following theorem, shows the exponential turnpike property of the Dirichlet control problems.

\begin{theorem}\label{mainresult3}
Let $\gamma \geq 0$ be a real number. Let $(u^T, g^T,\lambda^T)$ be the solution of \eqref{Dir-extremal-evol} and $(\overline{u},\overline{g},\overline{\lambda})$ the corresponding stationary solution of \eqref{Dir-extremal-sta}. Then, there is a constant $C=C(\gamma)>0$ (independent on $T$) such that for every $t\in[0,T]$ we have the following estimate:
\begin{align}\label{turnpike3}
\|u^T(\cdot,t)-\overline{u}\|_{H^{-s}(\Omega)}+ \|\lambda^T(\cdot,t)-\overline{\lambda}\|_{{H}^{-s}(\Omega)}\leq C\Big(e^{-\gamma t} + e^{-\gamma(T-t)}\Big)\Big(\|\overline{u}\|_{H^{-s}(\Omega)}+ \|\overline{\lambda}\|_{H^{-s}(\Omega)}\Big).
\end{align}
Moreover, the following estimate holds:
\begin{multline}\label{turnpike2}
\left\|\frac{1}{e^{-\gamma t}+e^{-\gamma(T-t)}}(u^T-\overline{u})\right\|_{L^2((0,T);H^{-s}(\Omega))}+\left\|\frac{1}{e^{-\gamma t}+e^{-\gamma(T-t)}}(g^T-\overline{g})\right\|_{L^2((0,T);L^2(\Omc))}\\
+ \left\|\frac{1}{e^{-\gamma t}+e^{-\gamma(T-t)}}(\lambda^T-\overline{\lambda})\right\|_{L^2((0,T);H^{-s}(\Omega))}\leq C\Big(\|\overline{u}\|_{H^{-s}(\Omega)}+\|\overline{\lambda}\|_{H^{-s}(\Omega)}\Big).
\end{multline}
\end{theorem}

\begin{proof}
We use the previous notations and take advantage of the proofs given in the previous sections.
We proceed in several steps.\\

{\bf Step 1:} Let us consider the functions $w:=u^T-\overline{u}$ and $\varphi:=\lambda^T-\overline{\lambda}$. Since the time derivatives of $\overline{u}$ and $\overline{\lambda}$ are zero, we can rewrite the system \eqref{Dir-extremal-sta} as follows:
\begin{align}\label{new-system}
\left(\begin{array}{cc}
-I & -\frac{d}{dt}+ \mathbb A\vspace*{0.1cm}\vspace*{0.1cm}\\
0 & E_T\vspace*{0.1cm}\\
\frac{d}{dt}+\mathbb A & \mathbb{B}\mathbb{B}^{*}\vspace*{0.1cm}\\
E_0 & 0
\end{array}\right)\left(\begin{array}{c}
\overline{u}\vspace*{0.1cm}\\
\overline{\lambda}
\end{array}\right)=\left(\begin{array}{c}
-u^d\vspace*{0.1cm}\\
\overline{\lambda}\vspace*{0.1cm}\\
0\vspace*{0.1cm}\\
\overline{u}
\end{array}\right).
\end{align}
Then, subtracting \eqref{new-system} from \eqref{Dir-extremal-evol}, we obtain that  $w$ and $\varphi$ satisfy
\begin{align}
\begin{cases}
w_t(t)+\mathbb{A}w(t)= -\mathbb{B}\mathbb{B}^{*}\varphi(t), & t>0\vspace*{0.1cm}\\
w(0)=w_0,&\vspace*{0.1cm}\\
-\varphi_t(t)+\mathbb{A}\varphi(t)=w(t), & t>0\vspace*{0.1cm}\\
\varphi(T)=\varphi_T,&
\end{cases}
\end{align}
where $w_0:=-\overline{u}$ and $\varphi_T:=-\overline{\lambda}$.

We  consider the following auxiliary problems:
\begin{align}\label{Dir-ii}
\begin{cases}
-\psi_t+\mathbb{A}\psi= 0, & \text{ in }[0,t]\vspace*{0.1cm}\\
\psi(t)=w(t),&
\end{cases}
\end{align}
and
\begin{align}\label{Dir-ii-2}
\begin{cases}
\xi_t+\mathbb{A}\xi=0, & \text{ in }[t,T]\vspace*{0.1cm}\\
\xi(t)=\varphi(t).&
\end{cases}
\end{align}
We observe that using the exponential stability of the semigroup, we can deduce that there  exist two constants $M,k>0$ such that 
\begin{align}
\|\psi(\cdot,\tau)\|_{H^{-s}(\Omega)}\leq Me^{-k(t-\tau)}\|w(\cdot,t)\|_{H^{-s}(\Omega)}, \quad 0\leq  \tau \leq t.
\end{align} 

Using the integration by parts formula \eqref{integration-by-parts}, multiplying the state equation for $w$ by $\psi$ (solution of \eqref{Dir-ii}), integrating over $(0,t)$, and the state equation for $\varphi$ by $\xi$ (solution of \eqref{Dir-ii-2}) and integrating over $(t,T)$, we obtain the following identities:
\begin{align}\label{Dir-i}
\|w(\cdot,t)\|_{H^{-s}(\Omega)}^2 = -\int_0^t ( \mathbb{B}^{*}\varphi(\cdot,\tau),\mathbb{B}^{*}\psi(\cdot,\tau))_{L^2(\Omc)}\;d\tau + ( w_0,\psi(\cdot,0))_{H^{-s}(\Omega)},
\end{align}
and 
\begin{align}\label{MJ}
\|\varphi(\cdot,t)\|_{H^{-s}(\Omega)}^2 = \int_t^T ( \xi(\cdot,\tau),w(\cdot,\tau))_{H^{-s}(\Omega)}\;d\tau + ( \xi(\cdot,T),\varphi_T)_{H^{-s}(\Omega)},
\end{align}
respectively.\\

{\bf Step 2:} Since $\mathbb{B}^{*}$ is an admissible observation operator for the semigroup $\mathbb T$, using \eqref{admissible-obs}, we can estimate the right  hand side of \eqref{Dir-i} with a constant independent of $t$.  Indeed, 
\begin{align*}
&\int_0^t\Big| ( \mathbb{B}^{*}\varphi(\cdot,\tau), \mathbb{B}^{*}\psi(\cdot,\tau))_{L^2(\Omc)} \Big|\;d\tau \\
&\leq \int_0^t  e^{-\frac{k}{2}(t-\tau)}\|\mathbb{B}^{*}\varphi(\cdot,\tau)\|_{L^2(\Omc)} \|\mathbb{B}^{*}e^{\frac{k}{2}(t-\tau)}\psi(\cdot,\tau)\|_{L^2(\Omc)}\;d\tau\\
&\leq \left(\int_0^t e^{-k(t-\tau)}\|\mathbb{B}^{*}\varphi(\cdot,\tau)\|_{L^2(\Omc)}^2\;d\tau\right)^{\frac 12} \left(\int_0^t\|\mathbb{B}^{*}e^{\frac{k}{2}(t-\tau)}\psi(\cdot,\tau)\|_{L^2(\Omc)}^2\;d\tau\right)^{\frac 12}.
\end{align*} 

It follows from \eqref{Dir-ii} that $\psi(\tau)=\mathbb{T}(t-\tau)w(\tau)$, for every $0\leq \tau\leq t$. Notice that the constant $k>0$ can be chosen such that the semigroup $(e^{\frac{k}{2}(t-\tau)}\mathbb{T}(t))_{t\geq 0}$ is still exponential stable. Thus, from \eqref{admissible-obs} we can deduce that
\begin{align*}
\left(\int_0^t\|\mathbb{B}^{*}e^{\frac{k}{2}(t-\tau)}\psi(\cdot,\tau)\|_{L^2(\Omc)}^2\;d\tau\right)^{\frac 12}&=\left(\int_0^t\|\mathbb{B}^{*}e^{\frac{k}{2}(t-\tau)}\mathbb{T}(t-\tau) w(\cdot,\tau)\|_{L^2(\Omc)}^2\;d\tau\right)^{\frac 12}\\
&\leq K\|w(\cdot,t)\|_{H^{-s}(\Omega)},
\end{align*}
where $K>0$ is independent of $t$. 

Using that 
\begin{align*}
( w_0,\psi(\cdot,0))_{H^{-s}(\Omega)}\leq \|w_0\|_{H^{-s}(\Omega)}\|w(\cdot,t)\|_{H^{-s}(\Omega)}M\sqrt{e^{-kt}}
\end{align*}
 and Young's inequality,  we can estimate \eqref{Dir-i} as follows:
\begin{align}\label{Dir-iii}
\|w(\cdot,t)\|_{H^{-s}(\Omega)}^2\leq C\left(\int_0^t e^{-k(t-\tau)}\|\mathbb{B}^{*}\varphi(\cdot,\tau)\|_{L^2(\Omc)}^2\;d\tau+\|w_0\|_{H^{-s}(\Omega)}^2 M^2e^{-kt} \right), \quad\forall t\in[0,T],
\end{align}
with a constant $C>0$ independent of $T$.

Using the fact that all the involved exponential functions are bounded by $1$, we obtain the following estimate for $w$ with respect to the norm of $C([0,T]; H^{-s}(\Omega))$:
\begin{align}\label{Dir-iv}
\|w\|_{C([0,T]; H^{-s}(\Omega))}^2\leq C\left(\|\mathbb{B}^{*}\varphi\|_{L^2((0,T);L^2(\Omc))}^2 +\|w_0\|_{H^{-s}(\Omega)}^2\right).
\end{align}

Similarly, we have the following estimate for $\varphi$:
\begin{align}\label{Dir-v}
\|\varphi\|_{C([0,T]; H^{-s}(\Omega))}^2\leq C\left( \|w\|_{L^2((0,T);H^{-s}(\Omega))}^2 +\|\varphi_T\|_{H^{-s}(\Omega)}^2\right).
\end{align}

Now, using \eqref{integration-by-parts} again, we have that
\begin{align}\label{Dir-vi}
&\int_0^T \left( \|\mathbb{B}^{*}\varphi(\cdot,\tau)\|_{L^2(\Omc)}^2 +\| w(\cdot,\tau)\|_{H^{-s}(\Omega)}^2 \right)\;d\tau\notag\\
&=( w_0, \varphi(\cdot,0))_{H^{-s}(\Omega)}- ( w(\cdot,T),\varphi_T)_{H^{-s}(\Omega)}\notag\\
&\leq \|w_0\|_{H^{-s}(\Omega)}\|\varphi(\cdot,0)\|_{H^{-s}(\Omega)} + \|w(\cdot,T)\|_{H^{-s}(\Omega)}\|\varphi_T\|_{H^{-s}(\Omega)}.
\end{align}

Applying Young's inequality to the right hand side in \eqref{Dir-vi}, we get that for every $\varepsilon_1,\varepsilon_2>0$,
\begin{multline*}
\|w_0\|_{H^{-s}(\Omega)}\|\varphi(\cdot,0)\|_{H^{-s}(\Omega)} + \|w(\cdot,T)\|_{H^{-s}(\Omega)}\|\varphi_T\|_{H^{-s}(\Omega)}\leq \frac{1}{2\varepsilon_1}\|w_0\|_{H^{-s}(\Omega)}^2 + \frac{\varepsilon_1}{2}\|\varphi(\cdot,0)\|_{H^{-s}(\Omega)}^2   \\ + \frac{1}{2\varepsilon_2}\|\varphi_T\|_{H^{-s}(\Omega)}^2 + \frac{\varepsilon_2}{2}\|w(\cdot,T)\|_{H^{-s}(\Omega)}^2.
\end{multline*}
Using \eqref{Dir-iv} and \eqref{Dir-v}, we can estimate the previous inequality as follows: 
\begin{multline*}
\|w_0\|_{H^{-s}(\Omega)}\|\varphi(\cdot,0)\|_{H^{-s}(\Omega)} + \|w(\cdot,T)\|_{H^{-s}(\Omega)}\|\varphi_T\|_{H^{-s}(\Omega)}\leq \frac{1}{2\varepsilon_1}\|w_0\|_{H^{-s}(\Omega)}^2 + \frac{1}{2\varepsilon_2}\|\varphi_T\|_{H^{-s}(\Omega)}^2 \\ + \frac{C\varepsilon_1}{2}\left( \|w\|_{L^2((0,T);H^{-s}(\Omega))}^2 +\|\varphi_T\|_{H^{-s}(\Omega)}^2\right) + \frac{C\varepsilon_2}{2}\left(\|\mathbb{B}^{*}\varphi\|_{L^2((0,T);L^2(\Omc))}^2 +\|w_0\|_{H^{-s}(\Omega)}^2\right).
\end{multline*}
Thus, \eqref{Dir-vi} can be estimated as follows:
\begin{align}\label{SS}
&\int_0^T \left( \|\mathbb{B}^{*}\varphi(\cdot,\tau)\|_{L^2(\Omc)}^2 +\| w(\cdot,\tau)\|_{H^{-s}(\Omega)}^2 \right)\;d\tau\notag\\
\leq& \frac{1}{2\varepsilon_1}\|w_0\|_{H^{-s}(\Omega)}^2 + \frac{1}{2\varepsilon_2}\|\varphi_T\|_{H^{-s}(\Omega)}^2+ \frac{C\varepsilon_1}{2}\left( \|w\|_{L^2((0,T);H^{-s}(\Omega))}^2 +\|\varphi_T\|_{H^{-s}(\Omega)}^2\right)\notag\\
& + \frac{C\varepsilon_2}{2}\left(\|\mathbb{B}^{*}\varphi\|_{L^2((0,T);L^2(\Omc))}^2 +\|w_0\|_{H^{-s}(\Omega)}^2\right).
\end{align}
Taking $\varepsilon_1,\varepsilon_2>0$ such that $1-\frac{C\varepsilon_1}{2}>0$ and $1-\frac{C\varepsilon_2}{2}>0$, we obtain from \eqref{SS} that there is a constant $C>0$ (independent of  $T$) such that
\begin{align}\label{Dir-vii}
\int_0^T \left( \|\mathbb{B}^{*}\varphi(\cdot,\tau)\|_{L^2(\Omc)}^2 +\| w(\cdot,\tau)\|_{H^{-s}(\Omega)}^2 \right)\;d\tau\leq C\Big(\|w_0\|_{H^{-s}(\Omega)}^2 + \|\varphi_T\|_{H^{-s}(\Omega)}^2\Big).
\end{align}

Combining \eqref{Dir-iv}-\eqref{Dir-v} and \eqref{Dir-vii} we can deduce that there is a constant  $C>0$ (independent of $T$) such that
\begin{align}
\|w\|_{C([0,T];H^{-s}(\Omega))}^2 + \|\varphi\|_{C([0,T];H^{-s}(\Omega))}^2 \leq C\Big(\|w_0\|_{H^{-s}(\Omega)}^2+\|\varphi_T\|_{H^{-s}(\Omega)}^2 \Big).
\end{align}

We have shown that the norm of the solution map 
$$\Lambda^{-1}:(H^{-s}(\Omega))^2 \to (C([0,T];H^{-s}(\Omega)))^2$$
 is bounded by a constant which is independent of $T$, where we recall that 
\begin{align}\label{Dir-Lambda}
\underbrace{\left(\begin{array}{cc}
-I & -\frac{d}{dt}+\mathbb A\vspace*{0.1cm}\\
0 & E_T\vspace*{0.1cm}\\
\frac{d}{dt}+\mathbb A & \mathbb{B}\mathbb{B}^{*}\vspace*{0.1cm}\\
E_0 & 0
\end{array}\right)}_{=:\Lambda}\left(\begin{array}{c}
w\\
\varphi
\end{array}\right)=\left(\begin{array}{c}
0\vspace*{0.1cm}\\
\varphi_T\vspace*{0.1cm}\\
0\vspace*{0.1cm}\\
w_0
\end{array}\right).
\end{align}
That is, there is a constant $C>0$ (independent of $T$ such that
\begin{align}\label{SS-W}
\|\Lambda^{-1}\|_{\mathcal L(H^{-s}(\Omega))^2,(C([0,T];H^{-s}(\Omega)))^2)}\le C.
\end{align}

{\bf Step 3:} We observe that the estimate \eqref{SS-W} for the map $\Lambda^{-1}$ is for solutions in $C([0,T]; H^{-s}(\Omega))$. Using Proposition \ref{prop-grune}, we can also derive an estimate on $L^2((0,T);H^{-s}(\Omega))$. Indeed, applying Proposition \ref{prop-grune} to the first term in the right hand side of \eqref{Dir-iii} we obtain that there is a constant $C>0$ (independent of $T$) such that 
\begin{align}\label{mjw}
\int_0^T\int_0^t e^{-k(t-\tau)}\|\mathbb{B}^{*}\varphi(\cdot,\tau)\|_{L^2(\Omc)}d\tau dt=&\left\|\left(\int_0^t e^{-k(t-\tau)}\|\mathbb{B}^{*}\varphi(\cdot,\tau)\|_{L^2(\Omc)}d\tau\right)\right\|_{L^1(0,T)}\notag\\
\leq &C\|\mathbb{B}^{*}\varphi\|_{L^2((0,T);L^2(\Omc))}^2.
\end{align}
Combing \eqref{Dir-iii}-\eqref{mjw}, we can deduce that there is a constant $C>0$ (independent of $T$) such that 
\begin{align}\label{Dir-viii}
\|w\|_{L^2((0,T);H^{-s}(\Omega))}^2 \leq C \left(\|\mathbb{B}^{*}\varphi\|_{L^2((0,T);L^2(\Omc))}^2 + \|w_0\|_{H^{-s}(\Omega)}^2\right).
\end{align}
In a similar way, we have that there is a constant $C>0$ (independent of $T$) such that 
\begin{align}\label{Dir-ix}
\|\varphi\|_{L^2((0,T);H^{-s}(\Omega))}^2 \leq C \Big(\|w\|_{L^2((0,T);H^{-s}(\Omega))}^2+ \|\varphi_T\|_{H^{-s}(\Omega)}^2\Big).
\end{align}

Combining \eqref{Dir-vii}-\eqref{Dir-viii} and \eqref{Dir-ix} we get the desired $L^2$--estimate for the solution map. That is, the norm of the operator
$$\Lambda^{-1}: (H^{-s}(\Omega))^2\to (L^2(0,T;H^{-s}(\Omega)))^2.$$
is bounded with a constant which is independent of $T$.
 Namely, there is a  constant $C>0$ (independent of $T$) such that 
\begin{align}\label{SS1}
\|\Lambda^{-1}\|_{\mathcal L((H^{-s}(\Omega))^2,(L^2(0,T;H^{-s}(\Omega)))^2)}\leq C.
\end{align}

{\bf Step 4:}
Proceeding as in Steps 3 and 4 in the proof of Theorem \ref{mainresult2}, we can deduce the turnpike property. Indeed, we choose $\gamma>0$ such that 
\begin{align}\label{theta}
\theta:=\gamma\|\Lambda^{-1}\|_{\mathcal{L}((H^{-s}(\Omega))^2, (L^2(0,T;H^{-s}(\Omega)))^2)}< 1.
\end{align}
 Define the functions 
\begin{align*}
\widetilde{w}:=\displaystyle\frac{1}{e^{-\gamma t}+ e^{-\gamma(T-t)}}w\;\mbox{ and }\;\widetilde{\varphi}:=\displaystyle\frac{1}{e^{-\gamma t}+ e^{-\gamma(T-t)}}\varphi.
\end{align*}
 Then,  $\widetilde{w}$ and $\widetilde{\varphi}$ satisfy
\begin{align}\label{final}
\left[\left(\begin{array}{cc}
-I & -\frac{d}{dt}+\mathbb A\vspace*{0.1cm}\\
0 & E_T\vspace*{0.1cm}\\
\frac{d}{dt}+\mathbb A & \mathbb{B}\mathbb{B}^{*}\vspace*{0.1cm}\\
E_0 & 0
\end{array}\right)-\gamma\underbrace{\left(\begin{array}{cc}
0 & -F\\
0 & 0\\
F & 0\\
0 & 0
\end{array}\right)}_{:=P}\right]
\left(\begin{array}{c}
\widetilde{w}\\
\widetilde{\varphi}
\end{array}\right)=\frac{1}{1+e^{-\gamma T}}\left(\begin{array}{c}
0\vspace*{0.1cm}\\
\varphi_T\vspace*{0.1cm}\\
0\vspace*{0.1cm}\\
w_0
\end{array}\right),
\end{align}
where 
$$F:=\frac{e^{-\gamma t}- e^{-\gamma(T-t)}}{e^{-\gamma t}+ e^{-\gamma(T-t)}}.$$ 
We observe that $\|P\|_{\mathcal L((L^2((0,T); H^{-s}(\Omega)))^2)}\leq 1$. Once again, letting 
$$\widetilde{\Phi}:=(\widetilde{w},\widetilde{\varphi})^T\; \mbox{ and }\; \mathcal{J}:=\frac{1}{1+e^{-\gamma T}}(w_0,\varphi_T)^T,$$ 
the system \eqref{final} is equivalent to
\begin{align*}
(I-\gamma \Lambda^{-1} P)\widetilde{\Phi}=\Lambda^{-1}\mathcal{J}.
\end{align*}
Since $\gamma\Lambda^{-1}P$ is a contraction, it follows from \cite[Theorem 2.14]{kress1989linear} that $I-\gamma\Lambda^{-1}P$ has a bounded inverse given by the Neumann series. That is,
\begin{align*}
(I-\gamma\Lambda^{-1}P)^{-1}=\sum_{k=0}^{\infty}(\gamma\Lambda^{-1}P)^{k}.
\end{align*}
In addition, we have that
\begin{align*}
\|(I-\gamma\Lambda^{-1}P)^{-1}\|_{\mathcal L((L^2((0,T);H^{-s}(\Omega)))^2)}\leq \frac{1}{1-\theta},
\end{align*}
where $0<\theta<1$ has been defined in \eqref{theta}.

From the previous computations, we can deduce that
\begin{align*}
\|\widetilde{w}\|_{L^2((0,T);H^{-s}(\Omega))}+\|\widetilde{\varphi}\|_{L^2((0,T);H^{-s}(\Omega))}\leq \frac{\|\Lambda^{-1}\|}{(1-\theta)(1+e^{-\gamma T})}\Big(\|w_0\|_{H^{-s}(\Omega)}+\|\varphi_T\|_{H^{-s}(\Omega)}\Big).
\end{align*}
Since $\frac{1}{1+e^{-\gamma T}}\leq 1$, and the norm of the solution map $\Lambda^{-1}$ was obtained in $L^2(0,T)$ and in $C([0,T])$, we can deduce the following estimates:
\begin{align}\label{L^2}
&\left\|\frac{1}{e^{-\gamma t}+e^{-\gamma(T-t)}}(u^T-\overline{u})\right\|_{L^2((0,T);H^{-s}(\Omega))} +\left \|\frac{1}{e^{-\gamma t}+e^{-\gamma(T-t)}}(\lambda^T-\overline{\lambda})\right\|_{L^2((0,T);H^{-s}(\Omega))}\notag\\
\leq &C\|\Lambda^{-1}\|_{\mathcal L((H^{-s}(\Omega))^2, (L^2(0,T;H^{-s}(\Omega)))^2)}(\|\overline{u}\|_{H^{-s}(\Omega)}+\|\overline{\lambda}\|_{H^{-s}(\Omega)})\notag\\
\le &C (\|\overline{u}\|_{H^{-s}(\Omega)}+\|\overline{\lambda}\|_{H^{-s}(\Omega)}),
\end{align}
where we have used \eqref{SS1}, and 
\begin{align}\label{C}
&\left\|\frac{1}{e^{-\gamma t}+e^{-\gamma(T-t)}}(u^T-\overline{u})\right\|_{C([0,T];H^{-s}(\Omega))} + \left\|\frac{1}{e^{-\gamma t}+e^{-\gamma(T-t)}}(\lambda^T-\overline{\lambda})\right\|_{C([0,T];H^{-s}(\Omega))}\notag\\
 \leq &C \|\Lambda^{-1}\|_{\mathcal L((H^{-s}(\Omega))^2,(C([0,T];H^{-s}(\Omega)))^2)}(\|\overline{u}\|_{H^{-s}(\Omega)}+\|\overline{\lambda}\|_{H^{-s}(\Omega)})\notag\\
\leq& C (\|\overline{u}\|_{H^{-s}(\Omega)}+\|\overline{\lambda}\|_{H^{-s}(\Omega)}),
\end{align}
where we have used \eqref{SS-W}.

From \eqref{C} we obtain the desired estimate \eqref{turnpike3} for the optimal state and adjoint vectors as follows:
\begin{align*}
\|u^T(\cdot,t)-\overline{u}\|_{H^{-s}(\Omega)}+\|\lambda^T(\cdot,t)-\overline{\lambda}\|_{H^{-s}(\Omega)}\leq C(\gamma)\Big(e^{-\gamma t}+e^{-\gamma(T-t)}\Big)\Big(\|\overline{u}\|_{H^{-s}(\Omega)}+\|\overline{\lambda}\|_{H^{-s}(\Omega)}\Big).
\end{align*}

{\bf Step 5:} To obtain an estimate for the control, we observe the following. Let $h:=g^T-\overline{g}$. Then 
\begin{align*}
\widetilde{h}:=\frac{1}{e^{-\gamma t}+e^{-\gamma(T-t)}}h=-\mathbb{B}^{*}\widetilde{\varphi}.
\end{align*} 
Therefore, there is a constant $C>0$ (independent of $T$) such that 
\begin{align}
\|\widetilde{h}\|_{L^2((0,T);L^2(\Omc))}^2 =&\int_0^T\|\mathbb{B}^{*}\widetilde{\varphi}(\cdot,t)\|_{L^2(\Omc)}^2 dt\notag\\
\leq& C\int_0^T\Big(\|\mathbb{B}^{*}\widetilde{\varphi}(\cdot,t)\|_{L^2(\Omc)}^2+ \|\widetilde{w}(\cdot,t)\|_{H^{-s}(\Omega)}^2\Big)dt.
\end{align}
Next, applying the integration by parts formula \eqref{integration-by-parts} to the solutions $\widetilde{w}$ and $\widetilde{\varphi}$, we can deduce that
\begin{align*}
\int_0^T\Big(\|\mathbb{B}^{*}\widetilde{\varphi}(\cdot,t)\|_{L^2(\Omc)}^2+ \|\widetilde{w}(\cdot,t)\|_{H^{-s}(\Omega)}^2\Big)dt=&( \widetilde{w}(\cdot,0),\widetilde{\varphi}(\cdot,0))_{H^{-s}(\Omega)} - (\widetilde{w}(\cdot,T),\widetilde{\varphi}(\cdot,T))_{H^{-s}(\Omega)} \\
&+\gamma\int_0^T \Big( (\widetilde{\varphi},F\widetilde{w})_{H^{-s}(\Omega)} + ( \widetilde{w}, F\widetilde{\varphi})_{H^{-s}(\Omega)} \Big)dt.
\end{align*}
Since $\|P\|_{L^2((0,T); H^{-s}(\Omega))^2}\leq 1$, it follows that 
\begin{align}\label{Dir-xi}
\int_0^T\Big(\|\mathbb{B}^{*}\widetilde{\varphi}(\cdot,t)\|_{L^2(\Omc)}^2&+ \|\widetilde{w}(\cdot,t)\|_{H^{-s}(\Omega)}^2\Big)dt\notag \\ 
\leq& \|\widetilde{w}(\cdot,0)\|_{H^{-s}(\Omega)}\|\widetilde{\varphi}(\cdot,0)\|_{H^{-s}(\Omega)} + \|\widetilde{w}(\cdot,T)\|_{H^{-s}(\Omega)}\|\widetilde{\varphi}(T)\|_{H^{-s}(\Omega)} \notag\\ 
&+ 2\gamma \|\widetilde{w}\|_{L^2((0,T);H^{-s}(\Omega))}\|\widetilde{\varphi}\|_{L^2((0,T);H^{-s}(\Omega))}\notag\\
\leq& \Big(\|\widetilde{w}(T)\|_{H^{-s}(\Omega)}+ \|\widetilde{\varphi}(0)\|_{H^{-s}(\Omega)}\Big)\Big(\|\widetilde{w}(0)\|_{H^{-s}(\Omega)} + \|\widetilde{\varphi}(T)\|_{H^{-s}(\Omega)}\Big) \notag\\ 
& +2\gamma\Big(\|\widetilde{w}\|_{L^2((0,T);H^{-s}(\Omega))}^2 + \|\widetilde{\varphi}\|_{L^2((0,T);H^{-s}(\Omega))}^2\Big).
\end{align}
From the definition of  $\widetilde{w}$, $\widetilde{\varphi}$, and the estimate of the norm of   $\Lambda^{-1}$ given in \eqref{SS-W},  we can estimate the first term in the right hand side of \eqref{Dir-xi} as follows: 
\begin{align*}
\|\widetilde{w}(\cdot,T)\|_{H^{-s}(\Omega)}+ \|\widetilde{\varphi}(\cdot,0)\|_{H^{-s}(\Omega)}&=\frac{1}{1+e^{-\gamma T}}\Big(\|w(\cdot,T)\|_{H^{-s}(\Omega)}+ \|\varphi(\cdot,0)\|_{H^{-s}(\Omega)}\Big)\\
&\leq \|\Lambda^{-1}\|_{\mathcal L((H^{-s}(\Omega))^2,C([(0,T]; H^{-s}(\Omega)))^2)} \Big(\|w_0\|_{H^{-s}(\Omega)}+ \|\varphi_T\|_{H^{-s}(\Omega)}\Big)\\
&\le  C\Big(\|w_0\|_{H^{-s}(\Omega)}+ \|\varphi_T\|_{H^{-s}(\Omega)}\Big),
\end{align*}
where the constant $C>0$ is independent of $T$.
Using  \eqref{L^2} we obtain the desired estimate for $g^T-\overline{g}$. That is, there is a constant $C=C(\gamma)>0$ (independent of $T$) such that
\begin{align}\label{U-L^2}
\left\|\frac{1}{e^{-\gamma t}+e^{-\gamma(T-t)}}(g^T-\overline{g})\right\|_{L^2((0,T);L^2(\Omc))}^2\leq C\Big(\|w_0\|_{H^{-s}(\Omega)}^2+\|\varphi_T\|_{H^{-s}(\Omega)}^2 \Big).
\end{align}
Combining \eqref{L^2}-\eqref{U-L^2}  we get \eqref{turnpike2} and the proof is finished.
\end{proof}

We conclude the paper with the following remark.

\begin{remark}
{\em We observe the following facts.
\begin{enumerate}
 \item Even if the solutions $u^T$, $\lambda^T\in L^2(\Omega\times(0,T))$ and $\overline{u}$, $\overline{\lambda}\in L^2(\Omega)$, we do not know if the estimates \eqref{turnpike3} and \eqref{turnpike2} can be replaced with the following estimates:
\begin{align}\label{feq}
\|u^T(\cdot,t)-\overline{u}\|_{L^2(\Omega)}+ \|\lambda^T(\cdot,t)-\overline{\lambda}\|_{L^2(\Omega)}\leq C\Big(e^{-\gamma t} + e^{-\gamma(T-t)}\Big)\Big(\|\overline{u}\|_{L^2(\Omega)}+ \|\overline{\lambda}\|_{L^2(\Omega)}\Big),
\end{align}
and
\begin{multline}\label{feq1}
\left\|\frac{1}{e^{-\gamma t}+e^{-\gamma(T-t)}}(u^T-\overline{u})\right\|_{L^2(\Omega\times (0,T))}+\left\|\frac{1}{e^{-\gamma t}+e^{-\gamma(T-t)}}(g^T-\overline{g})\right\|_{L^2((\Omc)\times(0,T))}\\
+ \left\|\frac{1}{e^{-\gamma t}+e^{-\gamma(T-t)}}(\lambda^T-\overline{\lambda})\right\|_{L^2(\Omega\times (0,T))}\leq C\Big(\|\overline{u}\|_{L^2(\Omega)}+\|\overline{\lambda}\|_{L^2(\Omega)}\Big),
\end{multline} 
respectively.  Of course in \eqref{turnpike3} and \eqref{turnpike2}, the term $\Big(\|\overline{u}\|_{H^{-s}(\Omega)}+ \|\overline{\lambda}\|_{H^{-s}(\Omega)}\Big),$ can be estimated as follows: There is a constant $C>0$ (depending only on $\Omega$, $N$ and $s$) such that  
\begin{align*}
\Big(\|\overline{u}\|_{H^{-s}(\Omega)}+ \|\overline{\lambda}\|_{H^{-s}(\Omega)}\Big)\leq C\Big(\|\overline{u}\|_{L^2(\Omega)}+ \|\overline{\lambda}\|_{L^2(\Omega)}\Big).
\end{align*}
The main difficulty occurs in the terms containing $(u^T-\overline{u})$.
Let us notice that even if we do not know how to prove \eqref{feq1}, it may be true, since, as we have already observed, $u^T$ and $\lambda^T$ belong to $L^2(\Omega\times(0,T))$. However, if \eqref{feq} holds, then this would imply  that  $u^T$ and $\lambda^T$ belong to $C([0,T];L^2(\Omega))$. We know that $\lambda^T$ enjoys this regularity but we do not know if $u^T$ has this regularity. We just know that $u^T\in C([0,T];H^{-s}(\Omega))$.

\item With the same reasons as in Remark \ref{rem-310}, we do not know if the estimate \eqref{turnpike3} can be improved as follows:
\begin{align*}
\|u^T(\cdot,t)-\overline{u}\|_{H^{-s}(\Omega)}+ \|\lambda^T(\cdot,t)-\overline{\lambda}\|_{{H}^{-s}(\Omega)}\leq C\Big(e^{-\gamma t}  \|\overline{u}\|_{H^{-s}(\Omega)}+ e^{-\gamma(T-t)} \|\overline{\lambda}\|_{H^{-s}(\Omega)}\Big).
\end{align*}
\end{enumerate}
}
\end{remark}

\noindent
{\bf Acknowledgement:} We would like to thank Umberto Biccari and Dario Pighin for their useful comments that helped to improve the quality of the manuscript.

\bibliographystyle{abbrv}
\bibliography{biblio}

\end{document}